\documentclass[11pt,a4paper]{article}

\usepackage[english]{babel}
\usepackage{graphicx} 
\usepackage{amsthm} 
\usepackage{amsmath}
\usepackage{amssymb}
\usepackage{amsthm} 
\usepackage{wasysym} 
\usepackage{mathrsfs} 
\usepackage{bbm} 
\usepackage{calligra} 
\usepackage{enumerate} 
\usepackage[shortlabels]{enumitem}
\usepackage{xcolor}
\usepackage{multirow} 

\newcommand{\N}{\mathbb{N}}
\newcommand{\Z}{\mathbb{Z}}

\newcommand{\R}{\mathbb{R}}

\newcommand{\T}{\mathbb{T}}
\newcommand{\Sp}{\mathbb{S}}

\newcommand{\Leb}[1]{L^{#1}}

\newcommand{\Cont}[1]{C^{#1}}

\newcommand{\Lebl}[1]{L^{#1}_{\mathrm{loc}}}

\newcommand{\Hil}[1]{H^{#1}}

\newcommand{\norma}[2]{\lVert#1\rVert_{#2}}

\newcommand{\Dis}{\mathcal{D}}
\newcommand{\Ker}{\mathrm{Ker}}

\newcommand{\pp}{\mathbf{p}}
\newcommand{\qq}{\mathbf{q}}
\newcommand{\LL}{\mathbf{S}}
\newcommand{\BB}{\mathbf{B}}
\newcommand{\CC}{\mathbf{C}}
\newcommand{\FF}{\mathbf{F}}
\newcommand{\PP}{\mathbf{P}}
\newcommand{\TT}{\mathbf{T}}
\newcommand{\Fspace}{\mathfrak{F}}

\newcommand{\function}[5]{\begin{eqnarray*}
		#1:#2 & \rightarrow & #3 \\
		   #4 & \mapsto     & #5
\end{eqnarray*}}

\newcommand{\velocity}{\mathbf{u}}
\newcommand{\aux}{\mathbf{m}}

\newcommand{\vect}{\mathbf{v}}
\newcommand{\map}{\mathbf{x}}
\newcommand{\mapy}{\mathbf{y}}

\newcommand{\Jacobian}{\mathbf{J}}

\newcommand{\BSO}{\mathbf{BS}}

\newcommand{\InterphaseMix}{\Omega_{\mathrm{mix}}}

\newcommand{\Kconstrain}{\mathcal{K}}

\newcommand{\Sub}{\mathbf{X}_0}
\newcommand{\SSub}{\mathbf{X}}
\newcommand{\Qfunc}[1]{\{#1\}_{\alpha}}
\newcommand{\Tfunc}{\mathscr{T}}
\newcommand{\Sfunc}{\mathscr{S}}
\newcommand{\Efunc}{\mathscr{E}}
\newcommand{\Dfunc}{\mathbf{D}}
\newcommand{\Gfunc}{\mathbf{G}}
\newcommand{\Hfunc}{\mathbf{H}}
\newcommand{\Velocity}{\mathbf{v}}
\newcommand{\Aux}{\mathbf{m}}
\newcommand{\Domain}{\mathscr{D}}
\newcommand{\DomainP}{{\mathscr{U}_{\mathrm{per}}}}
\newcommand{\DomainPy}{{\mathscr{U}'_{\mathrm{per}}}}
\newcommand{\expected}[1]{\langle #1\rangle}

\newcommand{\dif}{\,\mathrm{d}}

\newcommand{\Div}{\nabla\cdot}
\newcommand{\Curl}{\nabla^\perp\cdot}

\newcommand{\car}[1]{\mathbbm{1}_{#1}}

\newcommand{\dist}[2]{\mathrm{dist}(#1,#2)}

\newcommand{\almost}{\textrm{a.e.}\,}

\def\Xint#1{\mathchoice
	{\XXint\displaystyle\textstyle{#1}}%
	{\XXint\textstyle\scriptstyle{#1}}%
	{\XXint\scriptstyle\scriptscriptstyle{#1}}%
	{\XXint\scriptscriptstyle\scriptscriptstyle{#1}}%
	\!\int}
\def\XXint#1#2#3{{\setbox0=\hbox{$#1{#2#3}{\int}$}
		\vcenter{\hbox{$#2#3$}}\kern-.5\wd0}}

\def\dashint{\Xint-}  

\theoremstyle{theorem}
\newtheorem{thm}{Theorem}[section]
\newtheorem{prop}{Proposition}[section]
\newtheorem{cor}{Corollary}[section]
\newtheorem{lema}{Lemma}[section]
\theoremstyle{definition}
\newtheorem{defi}{Definition}[section]
\theoremstyle{remark}
\newtheorem{Rem}{Remark}[section]

\title{Degraded mixing solutions for the Muskat problem}
\author{\'A. Castro, D. Faraco, F. Mengual}

\begin{document}
	
\maketitle

\begin{abstract} We prove the existence of infinitely many mixing solutions for the Muskat problem in the fully unstable regime displaying a linearly degraded macroscopic behaviour inside the mixing zone. In fact, we estimate the volume proportion of each fluid in every rectangle of the mixing zone.
The proof is a refined version of the convex integration scheme presented in \cite{Onadmissibility,RIPM} applied to the subsolution in \cite{Mixing}. 
More generally, we obtain a quantitative h-principle for a class of evolution equations which shows that, in terms of weak*-continuous quantities, a generic solution in a suitable metric space essentially behaves like the subsolution. 
This applies of course to linear quantities, and in the case of IPM to the power balance $\PP$ \eqref{Power} which is quadratic.
As further applications of such quantitative h-principle we discuss the case of vortex sheet for the incompressible Euler equations.
\end{abstract}


\pagenumbering{arabic} 
\setcounter{page}{1}
\section{Introduction}\label{sec:Introduccion}

We study the dynamic of two incompressible 
fluids with constant densities $\rho^{\pm}$ and viscosities $\nu^{\pm}$, moving through a 2-dimensional porous media with permeability $\kappa$, under the action of gravity $\mathbf{g}=-(0,g)$. In this work we assume $\nu^{\pm}=\nu$, $\rho^+>\rho^-$ and $\kappa$ constant, and we denote $\vartheta=g\tfrac{\kappa}{\nu}$. This can be modelled (\cite{Mus37}) by the \textbf{IPM} (Incompressible Porous Media) system 
\begin{align}
\partial_t\rho+\nabla\cdot(\rho\velocity) & = 0, \label{IPM:1}\\
\Div\velocity & = 0,\label{IPM:2}\\
\tfrac{\nu}{\kappa}\velocity & = -\nabla p+\rho\mathbf{g}, \label{IPM:3}
\end{align}
in $\R^2\times(0,T)$,
where \eqref{IPM:1} represents the mass conservation
law, \eqref{IPM:2} the incompressibility, and \eqref{IPM:3} is Darcy's law, which relates the velocity of the fluid $\velocity$ with the forces (the pressure $p$ and the gravity $\mathbf{g}$) acting on it,
coupled with the \textbf{Muskat type initial condition} 
\begin{equation}\label{initialdensity}
\rho|_{t=0}=\rho_0=\rho^{+}\car{\Omega_{+}(0)}+\rho^{-}\car{\Omega_{-}(0)}.
\end{equation}
Without loss of generality we may assume $\rho^+=-\rho^-=\varrho>0$ (see sec. \ref{sec:CI}).
We focus on the situation when initially one of the fluids lies above the other, i.e., there is a function $f_0\in\Cont{1,\alpha}(\R)$ so that the initial interface is $\partial\Omega_{\pm}(0)=\mathrm{Graph}(f_0)=\{(s,f_0(s)): s\in\R\}$. 
The \textbf{Muskat problem} describes the evolution of the system \eqref{IPM:1}-\eqref{initialdensity} under the assumption that the fluids remain in contact at a moveable interface which divides $\R^2$ into two connected regions  $\Omega_{\pm}(t)$, i.e. $\partial\Omega_{\pm}(t)=\mathrm{Graph}(f(t))$, which turns out a Cauchy problem for $f$. 
If the heaviest fluid stays down, \textbf{fully stable regime}, 
such Cauchy problem is well-posed in Sobolev spaces (see \cite{Contour} for $\Hil{3}$ and \cite{CGS16,CGSV17,Mat16} for improvements of the regularity), whereas if the heaviest fluid stays on the top, \textbf{fully unstable regime}, it is ill-posed (see \cite{CCFL12} for $\Hil{4}$ and \cite{Contour} for $\Hil{s}$ with $s>3/2$ for small initial data). In spite of this, the existence of weak solutions of \eqref{IPM:1}-\eqref{initialdensity} in the fully unstable regime has been proved recently (see \cite{RIPM,Mixing,Piecewise} and also \cite{O}) by replacing the continuum free boundary assumption with the opening of a ``mixing zone'' where the fluids begin to mix ``indistinguishably'' (mixing solutions). In order to prevent misunderstanding we call the existence and properties of such mixing solutions the \textbf{Muskat-Mixing problem}. In \cite{Mixing} the authors define a \textbf{mixing zone} as
\begin{equation}\label{IMix}
\InterphaseMix(t)=\map(\R\times(-1,1),t),\quad t\in(0,T],
\end{equation}
and also $\InterphaseMix=\cup_{t\in(0,T]}\InterphaseMix(t)\times\{t\}$ from the map
\function{\map}{\R\times[-1,1]\times[0,T]}{\R^2}{(s,\lambda,t)}
{(s,f(s,t)+c\lambda t)}
where $f\in\Cont{}([0,T];\Hil{4}(\R))$ is a suitable evolution of $f_0\in H^5(\R)$ (see \cite[(1.11)]{Mixing}) and $c>0$ is the speed of growth of the mixing zone. At each $t\in(0,T]$, the mixing zone $\InterphaseMix(t)$
splits $\R^2$ into two open connected sets $\Omega_{\pm}(t)$ defined from
$\partial\Omega_{\pm}(t)=\map(\R,\pm 1,t)=\mathrm{Graph}(f(t))\pm (0,ct)$.
Notice
$\map\in\Cont{}((0,T];\mathrm{Diff}^1(\R\times(-1,1);\InterphaseMix(t)))$
because the Jacobian is $\Jacobian_{\map(t)}(s,\lambda)=ct$.
We recall the definition of mixing solution introduced in \cite{Mixing}.
\begin{defi}[Mixing solution, \cite{Mixing}]\label{mixing} A pair $\rho,\velocity\in\Leb{\infty}([0,T];\Leb{\infty}(\R^2))$
is a \textbf{mixing solution} for the map $\map$ if it is a weak solution of IPM (see \cite[def. 2.1]{Mixing}) satisfying:
\begin{enumerate}[(a')]
\item\label{DMS1o} 
$$\left\lbrace
\begin{array}{ll}
\rho=\pm\varrho & \textrm{a.e. in }\Omega_{\pm},\\
|\rho|=\varrho & \textrm{a.e. in }\InterphaseMix.
\end{array}\right.$$
\item\label{DMS2o} Mix in space-time: For every (space-time) open ball $B\subset\InterphaseMix$,
$$\int_{B}(\varrho-\rho(x,t))\dif x\dif t\int_{B}(\varrho+\rho(x,t))\dif x\dif t\neq0.$$
\end{enumerate} 
\end{defi}
The property \ref{DMS2o} predicts mixing in every (space-time) ball, but it does not give information about the volume proportion of each fluid. As it stands it does not exclude that  arbitrarily close to  $\Omega_+$ could be a sufficiently big ball with 99\% of $\rho_-$. In spite of the stochastic nature of the mixing phenomenon, this is obviously unrealistic from the experiments. In fact, as we shall explain after theorem \ref{thmprincipal}, we find natural to obtain mixing solutions displaying a linearly degraded macroscopic behaviour. In addition, we take care of replacing ``space-time'' by the stronger and more suitable version ``space at each time slice''.\\ 
\indent\textbf{The main result.} The aim of this paper is to prove that the Muskat-Mixing problem admits (infinitely many) solutions: continuous in time, mixing in space at each time slice and displaying a linearly degraded macroscopic behaviour. We call them degraded mixing solutions.
In fact, we have obtained an estimate of the volume proportion in every rectangle of $\InterphaseMix(t)$ at each time slice. Due to Lebesgue differentiation theorem, the error in this estimate depends on the size of the rectangles.
For the suitable definition we shall consider arbitrary $\alpha\in[0,1)$, space-error functions $\Sfunc\in\Cont{0}([0,1];[0,1])$ and time-error functions $\Tfunc\in\Cont{0}([0,T];[0,1])$ 
with $\Sfunc(0)=\Tfunc(0)=0$ and $\Sfunc(r),\Tfunc(t)>0$ for $r,t>0$. 
For instance
\begin{equation}\label{TSexample}
\Tfunc_{\epsilon}(t)=\epsilon e^{-\frac{1}{\epsilon}\left(\frac{1}{t}+t\right)},\quad
\Sfunc_{\varepsilon}(s)=\varepsilon e^{-\frac{1}{\varepsilon s}},
\end{equation} 
for $\varepsilon,\epsilon>0$ arbitrarily small.
We define 
\begin{equation}\label{Efuncdef}
\Efunc(\lambda,t)=\Sfunc(1-|\lambda|)\Tfunc(t),
\quad\Qfunc{A}=\frac{1\wedge|A|^\alpha}{|A|},
\end{equation}
where $(\lambda,t)\in(-1,1)\times(0,T]$, $|A|$ denotes the area of measurable sets $A$ in $\R^2$ and $\wedge$ the minimum between two quantities.\newpage
\begin{Rem}
The function $\Efunc$ has been introduced to show that the error in the estimate of the volume proportion also depends on the distance to the (space-time) boundary of the mixing zone. The parameter $\alpha$ has been introduced to refine this estimate for small rectangles. However, the case $\Efunc$ constant and $\alpha=0$ contains relevant information about the degraded mixing phenomenon, and it is easier to understand in a first reading.
\end{Rem}
\begin{defi}[Degraded mixing solution]\label{degradada} We say that $\rho,\velocity\in\Cont{0}([0,T];\Leb{\infty}_{w^*}(\R^2))$
is a \textbf{degraded mixing solution} for the map $\map$ of degree $(\alpha,\Efunc)$ if it is a weak solution of IPM such that, at each $t\in[0,T]$, it satisfies:
\begin{enumerate}[(a)]
\item\label{DMS1} 
$$\left\lbrace
\begin{array}{ll}
\rho(t)=\pm\varrho & \textrm{a.e. in }\Omega_{\pm}(t),\\
|\rho(t)|=\varrho & \textrm{a.e. in }\InterphaseMix(t).
\end{array}\right.$$
\item\label{DMS2} Mix in space at each time slice: For every non-empty (bounded) open $\Omega\subset\InterphaseMix(t)$, 
$$	\int_{\Omega}(\varrho-\rho(x,t))\dif x\int_{\Omega}(\varrho+\rho(x,t))\dif x\neq0.$$
\item\label{DMS3} Linearly degraded macroscopic behaviour: For every non-empty bounded rectangle $Q=S\times L\subset\R\times(-1,1)$,
$$\left|\dashint_{\map(Q,t)}\rho(x,t)\dif x-\expected{L}\varrho\right|\leq\Efunc(\expected{L},t)\Qfunc{\map(Q,t)}$$
where $\expected{L}=\dashint_L\lambda\dif\lambda\in(-1,1)$ is the center of mass of $L$.
\end{enumerate} 
\end{defi}
\begin{Rem}\label{contourlines} The new condition \ref{DMS3} implies a nice property. For $\lambda\in(-1,1)$ and $0<\delta<1$ consider the rectangle $Q_R^\delta(\lambda)=(-R,R)\times(\lambda-\frac{1}{R^\delta},\lambda+\frac{1}{R^\delta})$. This fits the contour line $\R\times\{\lambda\}$ when $R\rightarrow\infty$. Then, such degraded mixing solutions display a \textbf{perfect linearly degraded macroscopic behaviour on contour lines} $\map(\R,\lambda,t)$
\begin{equation}\label{perfectmixline}
\lim_{R\rightarrow\infty}\dashint_{\map(Q_R^\delta(\lambda),t)}\rho(x,t)\dif x=\lambda\varrho
\end{equation}
uniformly in $\lambda\in(-1,1)$ and $t\in(0,T]$.
\end{Rem}

Our main result is the following.

\begin{thm}\label{thmprincipal} Let $f_0\in\Hil{5}(\R)$, $\Efunc$ from \eqref{Efuncdef} and $\alpha\in[0,1)$. Then, the Muskat-Mixing problem for the initial interface $\mathrm{Graph}(f_0)$ and speed of growth $0<c<2\vartheta\varrho$ admits infinitely many degraded mixing solutions for the map $\map$ of degree $(\alpha,\Efunc)$.
\end{thm}

Let us explain in more detail the motivation and consequences of the theorem \ref{thmprincipal}. 
In \cite{O}, F. Otto proposed a relaxation approach of the Muskat-Mixing problem for the flat interface $f_0=0$. Roughly speaking, by neglecting the non-convex constraint $\rho\in\{\rho^-,\rho^+\}$ by its local average in space $\breve{\rho}\in[\rho^-,\rho^+]$, the author obtain an unique  solution
$$\breve{\rho}(x,t)=\left\lbrace\begin{array}{rl}
\pm\varrho, &  \pm x_2\geq 2\vartheta\varrho t, \\[0.1cm] 
\frac{x_2}{2\vartheta t}, &  |x_2|<2\vartheta\varrho t,
\end{array}\right.$$
of a relaxed problem (see also \cite{GO12}).
This can be thought as a solution of the Muskat-Mixing problem at a mesoscopic scale. 
In \cite{RIPM}, L. Sz\'ekelyhidi Jr. proved the existence of infinitely many weak solutions to the Muskat-Mixing problem for the initial flat interface $f_0=0$ satisfying \ref{DMS1o}. The proof is based on the convex integration method, which reverses Otto's relaxation switching from the ``\textbf{subsolution}'' $\breve{\rho}$ to exact solutions. In \cite{Mixing}, \'A. Castro, D. C\'ordoba and D. Faraco generalized Sz\'ekelyhidi's result for initial interfaces $f_0\in\Hil{5}(\R)$ and they added the property \ref{DMS2o}, where the subsolution is given by the following density function adapted to $\map$ 
\begin{equation}\label{densidadprincipal}
\breve{\rho}(x,t)=\left\lbrace\begin{array}{rl}
\pm\varrho, &  x\in\Omega_{\pm}(t), \\[0.1cm]
\lambda\varrho, &  x=\map(s,\lambda,t)\in\InterphaseMix(t).
\end{array}\right.
\end{equation}
Observe $\breve{\rho}$ can be though as a generalization of Otto's density for these more general initial interfaces.\\ 
If we understand this subsolution as a \textbf{coarse-grained density} for the Muskat-Mixing problem, it predicts a linearly degraded macroscopic behaviour of its solutions. 
Let us give a naive physical intuition about why such phenomenon is natural to be expected.
At molecular level, 
since the natural regime of the heaviest fluid is at the bottom, in the stable case the molecules are already 
well-placed, whereas
in the unstable case 
the molecules of the heaviest fluid are forced into break through the molecules of the lightest. Let us simplify the dynamic as a kind of random walk for the flat case to illustrate it (see \cite[sec. 2]{O} for a different approach). We interpret the conservation of mass and volume by setting that two close different molecules may interchange their positions if the heaviest is above the lightest, i.e., if their state is unstable due to gravity. Darcy's law is interpreted by setting that such interchange happens with some probability depending on the viscosities and in terms of the proximity to the rest molecules of the same fluid respectively. In the balanced case $\nu^+=\nu^-$, we set the probability is one half, independently of the relative position. We also set the size of the discretization $r=\triangle x_i=2c\triangle t$, where $c$ may depend on $\vartheta$ and $\varrho$. In spite of the randomness, the expectation on contour lines of molecules $\breve{\rho}_r$ is a deterministic function. In fact, this mean density $\breve{\rho}_r$ is an step function which decreases linearly from $\varrho$ to $-\varrho$  inside a strip (the mixing zone) that grows linearly in time with speed $c$. Returning to the continuum model, the coarse-grained density $\breve{\rho}$ \eqref{densidadprincipal} (for the flat case) is obtained
when $r\downarrow 0$.
\begin{figure}[h!]\centering
	\includegraphics[height=4.3cm]{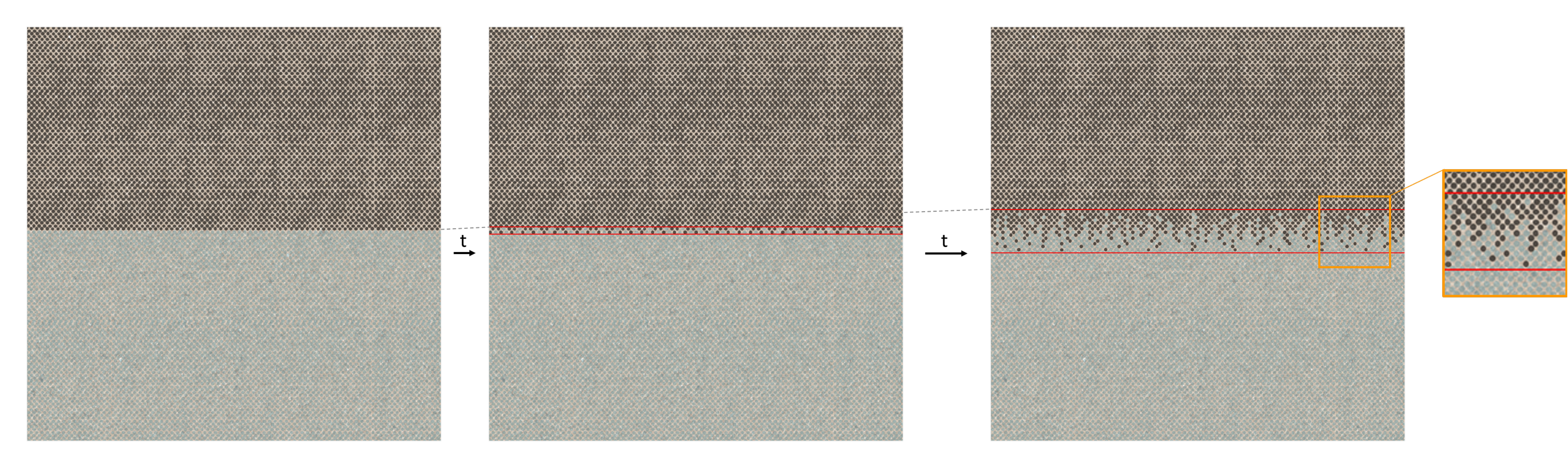}
	\caption{In the fully unstable regime, the molecules of the heaviest fluid are forced into break through the molecules of the lightest. The rate of expansion of the mixing zone and the coarse-grained density are linear.}
	\label{fig:molIzoom}\end{figure} 

\noindent This motivates to look for solutions $\rho$ of the Muskat-Mixing problem with a perfect linearly degraded macroscopic behaviour on contour lines (see rem. \ref{contourlines}). However, an error in the average between such $\rho$ and $\breve{\rho}$ is unavoidable on sufficiently small rectangles due to Lebesgue differentiation theorem. Since this error spreads as the molecules advance into the mixing zone, it must depend on the distance to where the fluids begin to mix too.\newpage
\noindent This is precisely the information recovered by theorem \ref{thmprincipal} (see fig. \ref{fig:dibujo}). Observe that the volume proportion of fluid with density $\rho^{\pm}$ in $\map(Q,t)$ is
\begin{equation}\label{ratefluid}
\frac{|\{x\in\map(Q,t)\,:\,\rho(x,t)=\pm\varrho\}|}{|\map(Q,t)|}=\frac{1}{2}\left(1\pm\frac{1}{\varrho}\dashint_{\map(Q,t)}\rho(x,t)\dif x\right),
\end{equation}
i.e., the average of $\rho$ quantifies the amount of each fluid. From \cite{Mixing} we know the existence of a sequence of mixing (in space-time) solutions $\rho_k$ such that
$\rho_k\overset{*}{\rightharpoonup}\breve{\rho}$. Thus, we would like to obtain solutions which are as close as possible to satisfy
\begin{equation}\label{aprox}
\dashint_{\map(Q,t)}\rho(x,t)\dif x
\approx
\dashint_{\map(Q,t)}\breve{\rho}(x,t)\dif x=\dashint_{L}\lambda\varrho\dif\lambda
=\expected{L}\varrho
\end{equation}
for every rectangle $Q=S\times L\subset\R\times(-1,1)$ at each  $t\in(0,T]$. However, Lebesgue differentiation theorem tells us that
\begin{equation}\label{LDT}
\lim_{\substack{|Q|\rightarrow 0 \\ \varsigma_0=\expected{Q},\,Q\textrm{ regular}}}\dashint_{\map(Q,t)}\rho(x,t)\dif x=\rho(x_0,t)
\end{equation}
for almost every $x_0=\map(\varsigma_0,t)\in\InterphaseMix(t)$ at each $t\in(0,T]$, where $\rho$ jumps unpredictably between $\pm\varrho$ because of \ref{DMS2}.
In other words, if the position is localized, $Q\downarrow\{\varsigma_0\}$, then the average of $\rho$ is undetermined. The opposite side of the coin is given by \ref{DMS3} because it states that we can know exactly the average of $\rho$ on unbounded domains. Schematically, this phenomenon can be interpreted as an ``uncertainty principle'' as follows:
\begin{center}
\begin{tabular}{cc|c|c|c|c|l}
\cline{3-4}
& & $Q$ & $\dashint_{\map(Q,t)}\rho(x,t)\dif x$ \\ \cline{1-4}
\multicolumn{1}{ |l  }{\multirow{2}{*}{Certainty} } &
\multicolumn{1}{ |l| }{Position} & $\{\varsigma_0\}$ & unpredictably     \\ \cline{2-4}
\multicolumn{1}{ |c  }{}                        &
\multicolumn{1}{ |l| }{Average} & unbounded & $\expected{L}\varrho$    \\ \cline{1-4}
\end{tabular}
\end{center}
As we have already commented, the first row is nothing but \ref{DMS2} in combination with \eqref{LDT}. The second row is due to \ref{DMS3} because such degraded mixing solutions satisfy
$$\lim_{\substack{L\downarrow L_0 \\ |S|\cdot|L|\rightarrow\infty}}\dashint_{\map(Q,t)}\rho(x,t)\dif x=\expected{L_0}\varrho$$
for every interval $L_0\subset(-1,1)$ at each $t\in(0,T]$. Consequently, the volume proportion of fluid with density $\rho^{\pm}$ in the strip $\map(\R,L_0,t)$ is exactly
$\frac{1}{2}(1\pm\expected{L_0})$.
Furthermore, theorem \ref{thmprincipal} not only quantifies these extremal situations, $Q\downarrow\{\varsigma_0\}$ and $Q\uparrow$ unbounded, but also the intermediate cases. More precisely, since $\varrho(1+|\expected{L}|)$ is the maximum possible error in \eqref{aprox} due to \eqref{LDT}, for every small $0<\varepsilon<\varrho$, such degraded mixing solutions improve the knowledge of \eqref{ratefluid} around each point $x_0=\map(\varsigma_0,t)\in\InterphaseMix(t)$ at each $t\in(0,T]$
\begin{equation}\label{knowledge}
\left|\dashint_{\map(Q,t)}\rho(x,t)\dif x-\expected{L}\varrho\right|\leq\varepsilon
\end{equation}
for every rectangle $Q=S\times L\subset\R\times(-1,1)$ containing $\varsigma_0$ in the regime
$$\Qfunc{\map(Q,t)}\leq\varepsilon\Efunc(\expected{L},t)^{-1}.$$
Observe $\alpha=1$ is excluded.\newpage
\noindent Moreover, for every small $\epsilon,\varepsilon>0$ there is $\delta>0$ such that, for all $Q=S\times L\subset\R\times(-1,1)$ with $(1-|\expected{L}|)\wedge|t|\leq\delta$ and $|\map(Q,t)|^{1-\alpha}\geq\epsilon$, then \eqref{knowledge} holds. That is, the uncertainty depends on the distance to the (space-time) boundary of the mixing zone.
In other words, the linearly degraded macroscopic behaviour is almost perfect close to where the fluids begin to mix.
\begin{figure}[h!]\centering
	\includegraphics[height=6cm]{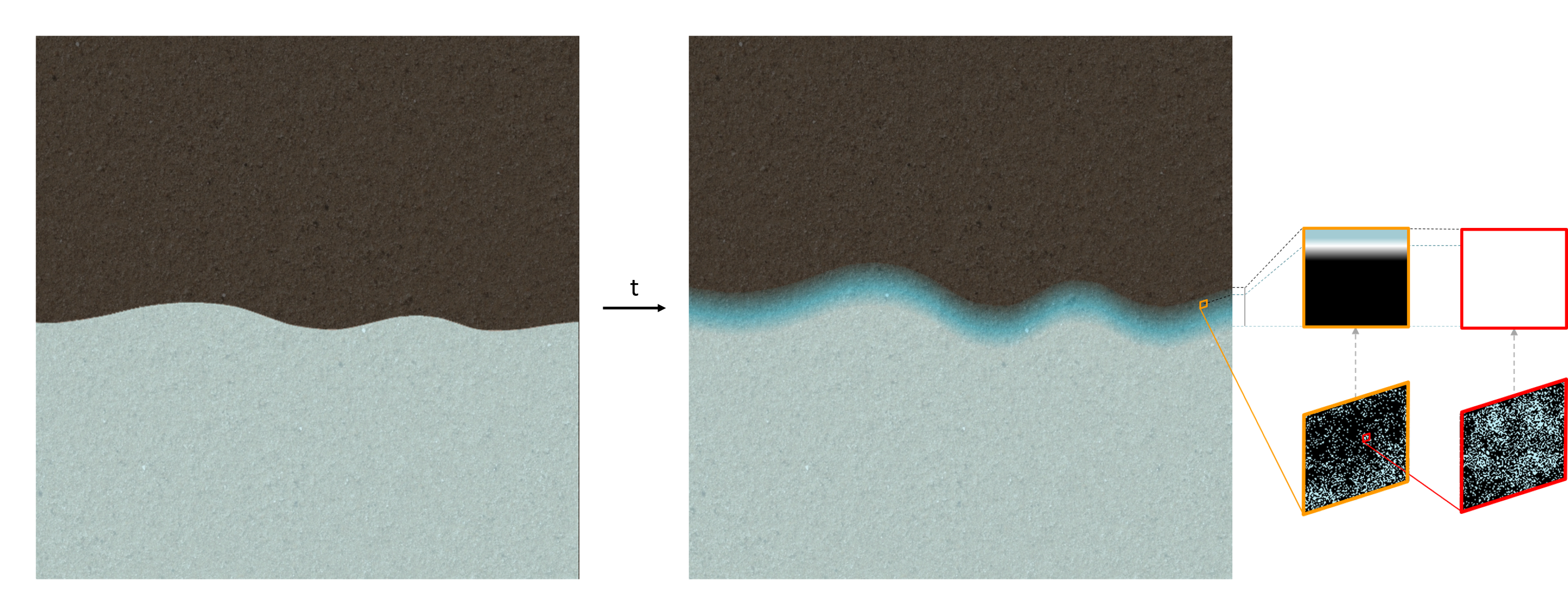}
	\caption{At a macroscopic scale, such degraded mixing solutions behave almost like the coarse-grained density (subsolution). The error (white) depends on the sizes of the rectangles and the distance to the (space-time) boundary of the mixing zone. Even for very small rectangles (orange), the volume proportion of the fluids is almost linearly distributed. In particular, the mixture is perfect on every contour strip. At a microscopic scale, the fluids are indistinguishable. Thus, for sufficiently small rectangles (red), the mass transport is unpredictably.}
	\label{fig:dibujo}\end{figure} 

Furthermore, the mass is not the only quantity that can be recovered from the subsolution. More precisely, our argument shows that linear quantities are almost preserved, e.g. that $\breve{\velocity}=\vartheta\BSO(-\partial_{1}\breve{\rho})$ (see \cite[sec. 4.1.2]{Mixing}) can be understood as the coarse-grained velocity. Moreover, not only linear quantities are inherited but also those weak*-continuous in the space of solutions of the stationary equations \eqref{IPM:2}\eqref{IPM:3} (see sec. \ref{sec:CI}). For instance, the ``power balance''
\begin{equation}\label{Power}
\PP(\rho,\velocity)=\velocity\cdot(\velocity+\vartheta(0,\rho))
=|\velocity|^2+\vartheta\rho\velocity_2,
\end{equation}
which can be interpreted as the balance between the density of energy per unit time consumed by the friction and the density of work per unit time done by the gravity. We have the following theorem.
\begin{thm}\label{thmprincipal2} In the context of theorem \ref{thmprincipal}, there exists infinitely many degraded mixing solutions such that, at each $t\in(0,T]$, it satisfies:
\begin{enumerate}[(d)]
\item\label{DMS4} For every non-empty bounded rectangle $Q=S\times L\subset\R\times(-1,1)$,
$$\left|\dashint_{\map(Q,t)}[\velocity-\breve{\velocity}](x,t)\dif x\right|\leq\Efunc(\expected{L},t)\Qfunc{\map(Q,t)}$$
and
$$\left|\dashint_{\map(Q,t)}[\PP(\rho,\velocity)-\PP(\breve{\rho},\breve{\velocity})](x,t)\dif x\right|\leq\Efunc(\expected{L},t)\Qfunc{\map(Q,t)}.$$
\end{enumerate} 
\end{thm}

Our proof is essentially based on the version of the h-principle presented by C. De Lellis and  Sz\'ekelyhidi Jr. for the incompressible Euler equations in \cite{Onadmissibility}. The concept of \textbf{h-principle} (homotopy principle) and the \textbf{convex integration} method was developed in Differential Geometry by M. Gromov (\cite{Gromov}) 
as a far-reaching generalization of the ground-breaking work of J. Nash (\cite{Nash}) and N. Kuiper (\cite{Kuiper}) for isometric embeddings. The philosophy of this method consists of adding suitable localized corrections to switch from some ``relaxed solution'' to exact solutions.
S. M$\ddot{\textrm{u}}$ller and V. $\check{\textrm{S}}$verak (\cite{MS03}) combined this method with Tartar compensated compactness (\cite{Tartar}) to apply it to PDEs and Calculus of Variations (see also \cite{Rigidity,DM}). 
Remarkably, De Lellis and  Sz\'ekelyhidi Jr. (\cite{Eulerinclusion}) discovered in 2009 that the incompressible Euler equations could also be brought to this framework, opening a new way to understand weak solutions in hydrodynamics, which end up in the proof of Onsager's conjecture (\cite{Ise16}).\\
In the context of IPM, C\'ordoba, Faraco and F. Gancedo (\cite{LIPM}) applied this method to prove lack of uniqueness.
In \cite{ActiveScalar} this result was extended to more general active scalar equations, and in \cite{IV15} the regularity of this kind of solutions was improved to $\Cont{\alpha}$. As we have already commented,  Sz\'ekelyhidi Jr. (\cite{RIPM}) proved the existence of infinitely many weak solutions to the Muskat-Mixing problem for the initial flat interface $f_0=0$ satisfying \ref{DMS1o}. In addition, Sz\'ekelyhidi Jr. showed that, for this relaxation, the mixing zone is always contained in a maximal mixing zone (\cite[prop. 4.3]{RIPM}) given by 
\begin{equation}\label{cmax}
c_{\max}=2\vartheta\varrho.
\end{equation}
As suggested in \cite{O} and \cite{RIPM}, 
in spite of the inherent stochasticity of the Muskat-Mixing problem explains the emanation of infinitely many microscopic solutions, there is a way to identify a selection criterion among subsolutions which leads to uniqueness, i.e.,
the physically relevant solutions are those which behave more like the mesoscopic solution.
In \cite{Mixing}, Castro, C\'ordoba and Faraco generalized Sz\'ekelyhidi's result for initial interfaces $f_0\in\Hil{5}(\R)$ and and they added the property \ref{DMS2o}.
In addition, they showed that, in the class of subsolutions given by \eqref{densidadprincipal}, the maximal speed of growth is \eqref{cmax} too.
Recently, C. F\"{o}rster and Sz\'ekelyhidi Jr. (\cite{Piecewise}) have proved the existence of mixing solutions for initial interfaces $f_0\in\Cont{3,\alpha}_*(\R)$. In particular, to attain \eqref{cmax} with the more manageable piecewise constant subsolutions, they constructed them as simple functions approaching $\breve{\rho}$. All this motivates the search of such degraded mixing solutions.\\
\indent This is the starting point of this paper. 
Since we want to control our solutions at each time slice, we follow \cite{Onadmissibility}.
This readily yields continuity in time, but with a careful look also the property (b). 
The third and more relevant aim is to prove the property (c). To this end, a more precise look at the h-principle of De Lellis and Sz\'ekelyhidi Jr. is required. 
Thus, the observation here is that, by defining more carefully the space of subsolutions, we can show that a generic solution will almost inherit the properties of the subsolution, which are described by weak*-continuous functionals.
For the Muskat-Mixing problem, these are the degraded mixing solutions. In fact, we have chosen to present a general theorem (thm. \ref{HP}) for a class of evolution equations, in the spirit of \cite{Onadmissibility,RIPM}, instead of an adapted version to the Muskat-Mixing problem. As an illustration, we shall discuss the case of vortex sheet for the incompressible Euler equations.\\

\indent The paper is organized as follows. In section \ref{sec:CI} we present the convex integration scheme and in section \ref{sec:HP} we prove the corresponding quantitative h-principle. This allows to prove theorems \ref{thmprincipal} and \ref{thmprincipal2} as particular cases in section \ref{sec:Proof}. In addition, we show an application to the vortex sheet problem in section \ref{sec:Other}.

\section{Convex integration scheme for a class of evolution equations}\label{sec:CI}

Before embarking in the more general scheme, let us recall the  case of the Muskat-Mixing problem to motivate it (\cite{LIPM,RIPM}). 
By applying the rotational operator on \eqref{IPM:3}, $p$ is eliminated 
\begin{align}
\Curl\velocity & = -\vartheta\partial_{1}\rho. \label{IPM:6}
\end{align}
By decomposing $\rho^{\pm}=\expected{\rho}\pm\varrho$ where
$\expected{\rho}=\tfrac{\rho^++\rho^-}{2}$ and 
$\varrho=\tfrac{\rho^+-\rho^-}{2}$
are the mean value and the deviation of the density respectively, it is clear we may assume  $\expected{\rho}=0$. 
More precisely, the mean density is absorbed by the pressure in \eqref{IPM:3}.
Thus, $\varrho$ is the significant term and all the results follow by adding $\expected{\rho}$ to $\rho$.\\ 
Now we normalize the problem as usual. Notice that if $(\rho,\velocity)$ is a degraded mixing solution with deviation $1$, parameter $1$ and initial density $\rho_0=\car{\Omega_{+}(0)}-\car{\Omega_{-}(0)}$, then the pair
$$\tilde{\rho}(x,t)=\varrho\rho(x,\vartheta\varrho t),
\quad\tilde{\velocity}(x,t)=\vartheta\varrho\velocity(x,\vartheta\varrho t),$$
is a degraded mixing solution with deviation $\varrho$, parameter $\vartheta$ and initial density \eqref{initialdensity}. 
Thus, from now on we may assume $\varrho=\vartheta=1$.\\ Following \cite{RIPM}, let us make the change of variables in $\R^2$ given by $\Velocity=2\velocity+(0,\rho)$
and let us introduce a new variable $\Aux$ to relax the non-linearity in IPM. Thus, our set of variables will be $z=(\rho,\Velocity,\Aux)\in\R\times\R^2\times\R^2\simeq\R^5$.
The expression of the stationary equations of IPM \eqref{IPM:2}\eqref{IPM:6}, the \textbf{BS} (Biot-Savart) system, in these variables is 
\begin{equation}\label{IPMR:2}
\begin{array}{r}
\Div(\Velocity-(0,\rho))=0\\
\Curl(\Velocity+(0,\rho))=0
\end{array}\quad\textrm{in }\Dis(\R^2)^*.
\end{equation}
We denote by $\Leb{\infty}_{\BSO}(\R^2)$ the closed linear subspace of $\Leb{\infty}_{w^*}(\R^2)$ consisting of functions $z=(\rho,\Velocity,\Aux)\in\Leb{\infty}(\R^2)$ satisfying \eqref{IPMR:2}. 
Therefore, by setting the constraint as usual
$$\Kconstrain=\left\{(\rho,\velocity,\aux)\in\R^5\,:\,\Aux=\tfrac{1}{2}(\rho\Velocity-(0,1)),\,|\rho|=1\right\},$$
the Muskat-Mixing problem attempts to find a bounded (with respect to the $\Leb{\infty}$-norm) and continuous curve in $\Leb{\infty}_{\BSO}(\R^2;\Kconstrain)$
$$z\in\Cont{0}([0,T];\Leb{\infty}_{\BSO}(\R^2;\Kconstrain)),$$
satisfying the Cauchy problem
\begin{equation}\label{IPMR:1}
\begin{array}{rl}
\partial_t \rho+\Div\Aux=0 & \textrm{in }\Dis(\R^2\times(0,T))^*,\\
z|_{t=0}=z_0 & \textrm{in }\Leb{\infty}_{\BSO}(\R^2;\Kconstrain),\\
\end{array}
\end{equation}
with $\rho_0=\car{\Omega_{+}(0)}-\car{\Omega_{-}(0)}$ and $\velocity_0=\BSO(-\partial_{1}\rho_0)$ (see \cite[(4.13)]{Mixing}).
In other words, the Muskat-Mixing problem can be written as a \textbf{differential inclusion}. In \cite{RIPM} it is observed that it is also convenient to consider some compact subsets of $\Kconstrain$. They are
$$\Kconstrain_M=\{(\rho,\Velocity,\Aux)\in \Kconstrain\,:\,|\Velocity|\leq M\}
\Subset\Kconstrain,\quad M>1.$$
The first step in the relaxation has been to replace the non-linearity with a new variable. Since this is too imprecise to capture the problem, the second step consists of restricting the variables to a bigger set $\tilde{\Kconstrain}\supset\Kconstrain$ for which the differential inclusion is still solvable and from which $\Kconstrain$ is ``reachable''.
In \cite{RIPM}, a suitable relaxation for the Muskat-Mixing problem is calculated. This is the $\Lambda$-convex hull (\cite[def. 4.3]{Rigidity}) of $\Kconstrain$ (and $\Kconstrain_M$). A point $z=(\rho,\velocity,\aux)$ belongs to $\Kconstrain^\Lambda$ if and only if it satisfies the inequalities
\begin{align}
|\rho|&\leq 1, \label{Khull:1}\\
\left|\aux-\tfrac{1}{2}\rho\Velocity\right|&\leq\tfrac{1}{2}(1-\rho^2), \label{Khull:2}
\end{align}
and it belongs to $\Kconstrain_M^\Lambda$ if and only if it satisfies
\eqref{Khull:1}\eqref{Khull:2} and
\begin{align}
|\Velocity|^2&\leq M^2-(1-\rho^2), \label{Khull:3}\\
\left|\aux-\tfrac{1}{2}\Velocity\right|&\leq\tfrac{M}{2}(1-\rho), \label{Khull:4}\\
\left|\aux+\tfrac{1}{2}\Velocity\right|&\leq\tfrac{M}{2}(1+\rho). \label{Khull:5}
\end{align}

\subsection{Tartar framework for evolution equations}

The Tartar framework is by now a well known approach to tackle non-linear equations arising in hydrodynamics in which the constitutive relations are interpreted as a differential inclusion.  We recall the definitions. As for the Muskat-Mixing problem, it seems to us convenient to distinguish between stationary equations and conservation laws. 
\begin{defi} Let $1<p_1,\ldots,p_N\leq\infty$ be H\"{o}lder exponents, $\Domain\subset\R^d$ a non-empty open domain and $\LL=(S_i)$ a $d$-tuple of $m_1\times N$ matrices. We denote by $\pp=(p_j)$ and $\Leb{\pp}_{\LL}(\Domain)$ the closed linear subspace of 
$$\Leb{\pp}_{w^*}(\Domain)=\bigotimes_{j=1}^N\Leb{p_j}_{w^*}(\Domain),$$ 
consisting of functions $z=(z_j)\in\Leb{\pp}(\Domain)$ satisfying the system of $m_1$ linear (stationary) equations
\begin{equation}\label{Linearproblem}
\LL\cdot\nabla z=\sum_{i=1}^dS_i\partial_i z=0
\quad\textrm{in }\Dis(\Domain)^*.
\end{equation}
If $\LL$ is trivial, simply $\Leb{\pp}_{\LL}(\Domain)=\Leb{\pp}_{w^*}(\Domain)$. If $p_1=\ldots=p_N=p$, we simply denote $\pp=p$. 
\end{defi}
Next we introduce the Cauchy problem. For some fixed closed (constraint) $\Kconstrain\subset\R^N$,
given an initial data $z_0\in\Leb{\pp}_{\LL}(\Domain;\Kconstrain)$, our aim is to find a bounded (with respect to the $L^\pp$-norm) and continuous curve in $\Leb{\pp}_{\LL}(\Domain;\Kconstrain)$
\begin{equation}\label{curvespace}
z\in\Cont{0}([0,T];\Leb{\pp}_{\LL}(\Domain;\Kconstrain)),
\end{equation}
satisfying the Cauchy problem
\begin{equation}\label{Cauchyproblem}
\begin{array}{rl}
C_0\partial_t z+\CC\cdot\nabla_x z=0 & \textrm{in }\Dis(\Domain\times(0,T))^*,\\
z|_{t=0}=z_0 & \textrm{in }\Leb{\pp}_{\LL}(\Domain;\Kconstrain),\\
\end{array}
\end{equation}
where $C_0=[I_{m_2}|0]\in \R^{m_2\times N}$ and $\CC=(C_i)$ is a $d$-tuple of $m_2\times N$ matrices ($m_2\leq N$).
We note that (the $\Leb{p}$-version of) \cite[lemma 8]{Onadmissibility} suggests that \eqref{curvespace} is an appropriate space for the problem.\\

In section \ref{sec:Hypothesis} we introduce several hypothesis, in the spirit of \cite{Onadmissibility,RIPM}, under which a quantitative h-principle shall be proved in section \ref{sec:HP}. Before setting them, let us give a brief explanation.\\
\indent \textbf{Plane-wave analysis.} For smooth solutions, the above system of $M=m_1+m_2$ linear equations can be written compactly in the original Tartar framework as 
\begin{equation}\label{LCproblem}
\TT\cdot\nabla z=0
\quad\textrm{in }\Domain\times(0,T),
\end{equation}
where $\TT=(T_i)$ is a $(d+1)$-tuple of $M\times N$ matrices by setting $T_i^{\mathrm{T}}=(S_i^{\mathrm{T}}|C_i^{\mathrm{T}})$
and $S_0=0$. Notice 
\eqref{LCproblem} can be also written in divergence-free form as
\begin{equation}\label{LCproblem1}
\Div(\TT z)=0
\quad\textrm{in }\Domain\times(0,T),
\end{equation}
where $\TT z=(T_0z|\cdots|T_dz)\in\R^{M\times(d+1)}$.  
The set of directions of one-dimensional oscillatory solutions of \eqref{LCproblem} is well-known as its \textbf{wave cone} (\cite{Tartar}) 
\begin{equation}\label{TW}
\Lambda_{\TT}=\bigcup_{\xi\in\Sp^{d-1}\times\R}\Ker(\TT\cdot\xi)=\left\{\bar{z}\in\R^N\,:\,\exists\xi\in\Sp^{d-1}\times\R\,\ni\,(\TT\bar{z})\xi=0\right\}
\end{equation}
where
$\TT\cdot\xi=\sum_{i=0}^d\xi_i T_i$.
We note that we are excluding the frequencies $\xi=(0,\xi_0)$ in \eqref{TW}. As will be discussed later, this is due to lemma \ref{lemmaYM}. \\
\indent \textbf{The relaxation.} As for the Muskat-Mixing problem, instead of focusing on the difficult problem (\ref{Cauchyproblem}), we consider also solutions of (\ref{Cauchyproblem}) for a suitable bigger set $\tilde{\Kconstrain}\supset\Kconstrain$ (briefly  (\ref{Cauchyproblem},$\tilde{\Kconstrain}$)) for which $\Leb{\pp}_{\LL}(\Domain;\tilde{\Kconstrain})$ is (weak*) closed, with the hope of going back to $\Kconstrain$ by adding localized one-dimensional oscillatory solutions of \eqref{LCproblem}. 
Thus, we say that $z\in\Cont{0}([0,T];\Leb{\pp}_{\LL}(\Domain;\tilde{\Kconstrain}))$ is a $\tilde{\Kconstrain}$-\textbf{subsolution} if it satisfies (\ref{Cauchyproblem},$\tilde{\Kconstrain}$). Obviously, $z$ is an \textbf{exact solution} if and only if $\tilde{\Kconstrain}=\Kconstrain$. 
The highly oscillatory behaviour of exact solutions is determined by the compatibility of the set $\Kconstrain$ with the cone $\Lambda_{\TT}$,
which is expressed in terms of the
$\Lambda_{\TT}$-convex hull of $\Kconstrain$,
$\Kconstrain^{\Lambda_{\TT}}$ (\cite[def. 4.3]{Rigidity}).  
Thus, a natural choice of $\tilde{\Kconstrain}$ is $\Kconstrain^{\Lambda_{\TT}}$.
However, sometimes it is enough to consider a smaller set (\cite{LIPM,ActiveScalar}). For a further explanation see \cite[sec. 4]{hprincipleFD}.\\
\indent \textbf{Convergence strategy.} 
We follow the strategy based on Baire category inspired in \cite{Rigidity}. Since we want to achieve the inclusion at every time our starting point
is \cite[sec. 3]{Onadmissibility}. To make general the arguments in \cite{Onadmissibility} we need a function which plays the role of distance function $\Dfunc$ which is ``semistrongly concave''.  
\begin{defi}\label{semistronglyconcave} 
We say that a concave function $\Dfunc\in\Cont{}(\R^N)$ is semistrongly concave if there are continuous functions $\Gfunc\in\Cont{}(\R^N;\R^N)$ and $0\neq\Hfunc\in\Cont{}(\R^N;\R_+)$ being $\Hfunc$ positive homogeneous of degree $\gamma\geq 1$ such that
$$\Dfunc(z+w)\leq \Dfunc(z)+\Gfunc(z)\cdot w-\Hfunc(w),\quad z,w\in\R^N.$$
\end{defi}
\begin{Rem} This concept can be understood as a weakening of the classical strongly concavity, for which we recall $\Hfunc$ must be $\Hfunc(w)=C|w|^2$ for some $C>0$. On the one hand, this notion admits directions where $\Hfunc$ can be zero, e.g. $(v,u)\mapsto e-\tfrac{1}{2}|v|^2$ in \cite{Onadmissibility}. On the other hand, it does not require the Hessian to be uniformly definite negative, e.g.
$v\mapsto e-|v|^\gamma$ for $\gamma>2$. 
\end{Rem}
\indent\textbf{Quantitative h-principle.} At the end, we shall reverse the relaxation. This h-principle can be written schematically in the standard way as
$$\begin{array}{ccccc}
& (\ref{Cauchyproblem},\Kconstrain) & \overset{\textrm{relaxation}}{\longrightarrow} & (\ref{Cauchyproblem},\tilde{\Kconstrain}) & \\
& \rotatebox[origin=c]{-90}{$\displaystyle\dashrightarrow$} & \textrm{h-principle} & \rotatebox[origin=c]{-90}{$\displaystyle\longrightarrow$} & \\
\textrm{solution} & z & \underset{\substack{\textrm{convex} \\ \textrm{integration}}}{\longleftarrow} & \breve{z} & \tilde{\Kconstrain}\textrm{-subsolution}
\end{array}$$
Since we want to select those solutions which best emphasize the relation between the relaxation and the convex integration, we introduce a family of (weak*) continuous functionals to test the h-principle diagram, in other words, to help us restrict the space of $\tilde{\Kconstrain}$-subsolutions. Let us explain it in more detail. Notice the property \ref{DMS3} for the Muskat-Mixing problem requires to bound
\begin{equation}\label{Fgexplanation}
\int_{\map(Q,t)}[\FF(z)-\FF(\breve{z})](x,t)g(x,t)\dif x
\end{equation}
where $\FF(z)=\rho$ for $z=(\rho,\Velocity,\Aux)$ and $g=1$. For the general case \eqref{Cauchyproblem}, $\FF$ can be chosen as any affine transformations or, more generally, as any functional $\FF:\Leb{\pp}_{\LL}(\Domain;\tilde{\Kconstrain})\rightarrow\Lebl{p}(\Domain)$ (weak*) continuous on bounded subsets of $\Domain$ for some $p\in(1,\infty]$ (e.g. $\FF(z)=\velocity$ and $\FF=\PP$ for the property \ref{DMS4} in thm. \ref{thmprincipal2}). The $\Leb{p}$-duality suggests to consider weights $g\in\Cont{}((0,T];\Lebl{q}(\Domain))$ with $q\in(p^*,\infty]$ ($\frac{1}{p}+\frac{1}{p^*}=1$). We note that we are excluding $q=p^*$. The reason is to leave room for another H\"{o}lder exponent $r$ in order to prove a suitable perturbation property as in \cite[prop. 3]{Onadmissibility}. Although for the Muskat-Mixing problem we only need $g=1$, we know that other problems shall require this generality.
For the map, we need first an open set $\DomainP$ in the space-time domain $\Domain\times(0,T]$ with\begin{align*}
\DomainP(t)&=\{x\in\Domain\,:\, (x,t)\in \DomainP\}\neq\emptyset,\quad t\in(0,T],
\end{align*}
which plays the roll of $\InterphaseMix(t)$ for the Muskat-Mixing problem. In other words, this will be the domain where the $\tilde{\Kconstrain}$-subsolution fails to be exact but being ``perturbable''. Then, we consider an auxiliary open set $\DomainPy$ in $\R^d\times(0,T]$ and some change of variables $\mapy:\DomainPy\rightarrow\Domain$ with
$\mapy\in\Cont{}((0,T];\textrm{Diff}^{1}(\DomainPy(t);\DomainP(t)))$ which play the roll of $\R\times(-1,1)\times(0,T]$ and the map $\map$ respectively for the Muskat-Mixing problem. Let us set the definition of this family.
\begin{defi} We define $\Fspace$ as the family consisting of triples  $(\FF,g,\mapy)$ satisfying: 
\begin{itemize}
\item[($\FF,g$)] There is a triple $(p,q,r)\in[1,\infty]^3$ with $r\neq\infty$ and $\frac{1}{p}+\frac{1}{q}+\frac{1}{r}=1$ so that $\FF:\Leb{\pp}_{\LL}(\Domain;\tilde{\Kconstrain})\rightarrow\Lebl{p}(\Domain)$, $g:\Domain\times(0,T]\rightarrow\R$ and, for every bounded open $\Omega\subset\Domain$,
$$\FF\in\Cont{}(\Leb{\pp}_{\LL}(\Omega;\tilde{\Kconstrain});\Leb{p}_{w^*}(\Omega)),\quad g\in\Cont{}((0,T];\Leb{q}(\Omega)).$$
\item[($\mapy$)] There is an open set $\DomainPy$  in the space-time domain $\R^d\times(0,T]$ so that  $\mapy:\DomainPy\rightarrow\Domain$ with
$$\mapy\in\Cont{}((0,T];\textrm{Diff}^{1}(\DomainPy(t);\DomainP(t))).$$
\end{itemize}
\end{defi}
\begin{Rem}
Notice that $\Fspace$ is non-empty since we can always consider trivially $\DomainPy=\DomainP$ and $\mapy(t)=\mathrm{id}_{\DomainP(t)}$. 
\end{Rem}

\subsection{Hypothesis}\label{sec:Hypothesis}

From now on we assume that there are a closed set $\tilde{\Kconstrain}\supset\Kconstrain$ for which $\Leb{\pp}_{\LL}(\Domain;\tilde{\Kconstrain})$ is (weak*) closed and some open set $U\subset\tilde{\Kconstrain}\setminus\Kconstrain$  such that the following three hypothesis holds (\cite{Onadmissibility,RIPM}).\\

\textbf{(H1) The wave cone}. There are a cone $\Lambda\subset\R^N$ and a profile $0\neq h\in\Cont{1}(\T;[-1,1])$ with $\int h=0$ such that the following holds. For all $\bar{z}\in\Lambda$ and $\psi\in\Cont{\infty}_c(\R^{d+1})$ there exists $\xi\in\Sp^{d-1}\times\R$ so that there are localized smooth solutions of \eqref{LCproblem} of the form
$$\tilde{z}_k(y)=\bar{z}h(k\xi\cdot y)\psi(y)+\mathcal{O}(k^{-1})$$
where $\mathcal{O}$ only depends on $|\bar{z}|$, $|\xi|$ 
and $\{|D^\alpha\psi(y)|\,:\,1\leq|\alpha|\leq n\}$ for some fixed $n$.\newpage

\textbf{(H2) The $\Lambda$-segment}. There is a semistrongly concave function $\Dfunc\in\Cont{}(\R^N)$ and an increasing function $\Phi\in\Cont{}((0,1];(0,\infty))$ such that the following holds.
\begin{itemize}
\item[($\Dfunc$)] The restriction function satisfies $\Dfunc(\tilde{\Kconstrain})\subset[0,1]$ with $\Dfunc|_{\tilde{\Kconstrain}}^{-1}(0)\subset\Kconstrain.$ 
\item[($\Phi$)] For all $z\in U$ there is a large enough $\Lambda$-direction $\bar{z}\in\Lambda$ $$\Hfunc(\bar{z})\geq\Phi(\Dfunc(z))$$ 
where $\Hfunc$ is given in def. \ref{semistronglyconcave}, while the associated $\Lambda$-segment stays in $U$
$$z+[-\bar{z},\bar{z}]\subset U.$$
\end{itemize}

\textbf{(H3) The space of $\tilde{\Kconstrain}$-subsolutions}. There exists a $\tilde{\Kconstrain}$-subsolution $\breve{z}$ which is exact outside $\DomainP$
$$\breve{z}(x,t)\in\Kconstrain\quad\almost x\in\Domain\setminus\DomainP(t),\,\forall t\in(0,T],$$
while it is perturbable inside
$$\breve{z}\in\Cont{}(\DomainP;U).$$
We surround $\breve{z}$ in a topological space $\Sub$ consisting of admissible perturbations of $\breve{z}$. Here, $z$ perturbable means that $z$ is continuous from $\DomainP$ to $U$, and an admissible perturbation of $z$ is $\tilde{z}_k$ such that $z_k=z+\tilde{z}_k$ is also perturbable. Obviously, we impose as usual $z$ and $z_k$ to be $\tilde{\Kconstrain}$-subsolutions. The new feature is that we require $z$ and $z_k$ to be ``close'' to $\breve{z}$ in the sense that they make \eqref{Fgexplanation} small. Next we give the precise definition of $\Sub$. For that we consider a finite family $\Fspace_0\subset\Fspace$ and, for each $(\FF,g,\mapy)\in\Fspace_0$, we fix some $0<\alpha<\frac{1}{r}$, some space-error function $\Sfunc\in\Cont{0}([0,1];[0,1])$ and some time-error function $\Tfunc\in\Cont{0}([0,T];[0,1])$ with $\Sfunc(0)=\Tfunc(0)=0$ and $\Sfunc(s),\Tfunc(t)>0$ for $s,t>0$ (e.g. \eqref{TSexample}). We define $$\Efunc(\varsigma,t)=\Sfunc(1\wedge\dist{\varsigma}{\partial\DomainPy(t)})\Tfunc(t),
\quad\Qfunc{A}=\frac{1\wedge|A|^\alpha}{|A|},$$
where $(\varsigma,t)\in\R^d\times(0,T]$ and $|A|$ denotes the volume of measurable sets $A$ in $\R^d$.
Then, we define $\Sub$ as follows.
\begin{defi}\label{spacesubsolutions} A $\tilde{\Kconstrain}$-subsolution $z$ belongs to $\Sub$ if it satisfies the following conditions.
\begin{enumerate}
\item[($U$)]\label{H3U} It agrees with $\breve{z}$ where the constraint holds
$$z(x,t)=\breve{z}(x,t)\quad\almost x\in\Domain\setminus\DomainP(t),\,\forall t\in[0,T],$$
while it is perturbable where the constraint fails
$$z\in\Cont{}(\DomainP;U).$$
\item[($\Fspace_0$)]\label{CIrestrictive} There is $C(z)\in(0,1)$ such that, for all $(\FF,g,\mapy)\in\mathfrak{F}_0$,
$$\left|\dashint_{\mapy(Q,t)}\left[\FF(z)-\FF(\breve{z})\right](x,t)g(x,t)\dif x\right|\leq C(z)\Efunc(\expected{Q},t)\Qfunc{\mapy(Q,t)}$$
for every non-empty (non-necessarily regular) bounded cube $Q\subset\DomainPy(t)$ and $t\in(0,T]$. Here, 
$\expected{Q}=\dashint_{Q}x\dif x$ is the center of mass of $Q$.
\end{enumerate}
We say that an element of $\Sub$ is a $(\breve{z},\Fspace_0)$\textbf{-subsolution} (notice $\breve{z}\in\Sub$).
\end{defi}
Once we have defined $\Sub$, we make the following assumption. There is a closed ball $\BB$ of $\Leb{\pp}(\Domain)$ such that, every $\tilde{\Kconstrain}$-subsolution $z$ satisfying ($U$) is a curve inside $\BB$, $z([0,T])\subset\BB$. Indeed, $z([0,T])\subset\tilde{\BB}=\BB\cap\Leb{\pp}_{\LL}(\Domain;\tilde{\Kconstrain})\Subset\BB$. 
Under this assumption, we define the closure of $\Sub$ in the standard way as in \cite{Onadmissibility}.
Since $\BB$ is compact and metrizable in $\Leb{\pp}_{w^*}(\Domain)$, 
its metric $d_{\BB}$ induces naturally a metric $d$ on $\mathbf{Y}=\Cont{0}([0,T];(\tilde{\BB},d_{\BB}))$ via
$$d(z,w)=\sup_{t\in[0,T]}d_{\BB}(z(t),w(t)),\quad z,w\in\mathbf{Y}.$$
The space $\mathbf{Y}$ inherits the completeness of $\tilde{\BB}$. 
The topology induced by $d$ on $\mathbf{Y}$ is equivalent to the topology of $\mathbf{Y}$ as subset of $\Cont{0}([0,T];\Leb{\pp}_{\LL}(\Domain;\tilde{\Kconstrain}))$. We define $\SSub$ as the closure of $\Sub$ in $(\mathbf{Y},d)$. In this way, $\SSub$ is a complete metric subspace of $\mathbf{Y}\subset \Cont{0}([0,T];\Leb{\pp}_{\LL}(\Domain;\tilde{\Kconstrain}))$.


\begin{Rem} There are some differences between the Muskat-Mixing problem and the incompressible Euler equations considered in \cite{Onadmissibility}. On the one hand, for the Muskat-Mixing problem the domain $\DomainP(t)$ depends on time, whereas in \cite{Onadmissibility} does not. In addition, this convex integration scheme gives more information about the solutions thanks to ($\Fspace_0$). Indeed, for $\Fspace_0$ trivial we would recover the usual scheme.	
On the other hand, in \cite{Onadmissibility} the constraint $\Kconstrain(x,t)$ and consequently the sets $\tilde{\Kconstrain}(x,t)$ and $U(x,t)$ depend on space-time. Although we could combine both cases, we prefer this approach for 
convenience and simplicity.
\end{Rem}

\section{Quantitative h-principle for a class of evolution equations}\label{sec:HP}

In this section we prove a quantitative h-principle assuming \textrm{(H1)-(H3)} from section \ref{sec:Hypothesis}.
First of all let us recall several notions in Baire category theory (\cite{Oxt80}). Given a complete metric space $\SSub$, a set $R\subset\SSub$ is \textbf{residual} if it is countable intersection of open dense sets. By virtue of Baire category theorem, every residual set is dense. A function $\mathcal{J}:\SSub\rightarrow\R$ is \textbf{Baire-1} if it is pointwise limit of continuous functions, e.g. if $\mathcal{J}$ is upper-semicontinuous ($\limsup_{z\rightarrow z_0}\mathcal{J}(z)\leq\mathcal{J}(z_0)\,\forall z_0\in\SSub$).
The set of continuity points of a Baire-1 function $\mathcal{J}$
$$\SSub_{\mathcal{J}}=\left\{z\in\SSub\,:\, \mathcal{J}\textrm{ is continuous at }z\right\}$$
is residual in $\SSub$. 
Now, let $\SSub$ as in \textrm{(H3)}. Since $\DomainP$ is open in $\R^d\times(0,T]$, for every $(x_0,t_0)\in\DomainP$ there are a bounded open domain (with Lipschitz boundary) $x_0\in\Omega\Subset\DomainP(t_0)$ and a time-interval $I=[t_1,t_2]\Subset(0,T]$ with $t_1<t_0\leq t_2$ such that $\Omega\times I\Subset\DomainP$.
We associate to $\Omega\times I$ the \textbf{relaxation-error functional}
\function{\mathcal{J}}{\mathbf{X}}{\R_+}{z}
{\sup_{t\in I}\int_{\Omega}\Dfunc(z(x,t))\dif x}
well defined because $\Leb{\pp}_{\LL}(\Domain;\tilde{\Kconstrain})$ is closed by definition of $\tilde{\Kconstrain}$ and  \textrm{(H2,$\Dfunc$)}. Furthermore, \textrm{(H2,$\Dfunc$)} also implies
\begin{equation}\label{J0}
\mathcal{J}^{-1}(0)\subset\left\{z\in\SSub\,:\,z(x,t)\in\Kconstrain
\quad\almost x\in \Omega,\,\forall t\in I\right\}.
\end{equation}
We omit the proof of the following lemma since it is analogous to \cite[lemma 4]{Onadmissibility}. The crucial information here is that $\Dfunc$ is concave and bounded on $\tilde{\Kconstrain}$. 

\begin{lema}\label{semicontinua} The functional $\mathcal{J}$ is upper-semicontinuous. Therefore, $\SSub_{\mathcal{J}}$ is residual in $\SSub$.
\end{lema}

\newpage
The following lemma, which is nothing but a simple observation in Young measure theory, generalizes \cite[lemma 7]{Onadmissibility}. Observe this can be understood as a generalization of Riemann-Lebesgue lemma. For our purpose, since the convex integration method is based on adding suitable perturbations $\tilde{z}_k$ from \textrm{(H1)} to a given $z\in\Sub$, for $h$ as in \textrm{(H1)} and $A=\mathrm{id}$ this lemma will imply $z_k=z+\tilde{z}_k\overset{d}{\rightarrow}z$, whereas for $A=\Hfunc$ as in \textrm{(H2,$\Dfunc$)} (recall def. \ref{semistronglyconcave}) it will imply $\mathcal{J}(z_k)\nrightarrow\mathcal{J}(z)$. Notice the same cannot be done with frequencies $\xi=(0,\xi_0)$. All this allow to prove $\SSub_{\mathcal{J}}\subset\mathcal{J}^{-1}(0)$ and then, by covering $\DomainP$, our quantitative h-principle. 
\begin{lema}\label{lemmaYM} Let $h\in\Leb{\infty}(\T;\R^N)$ and $\xi\in\Sp^{d-1}\times\R$. Then, for every open $\Omega\subset\R^d$, $g\in\Leb{1}(\Omega)$ and $A\in\Cont{}(\R^N)$, 
\begin{equation}\label{lemmaYM1}
\int_{\Omega}g(x)A(h(k\xi\cdot(x,t)))\dif x
\rightarrow\int_{\Omega}g(x)\dif x\int_{\T}A(h(\tau))\dif\tau
\end{equation}
uniformly in $t\in\R$ when $k\rightarrow\infty$.
\end{lema}
\begin{proof} First assume $g\in\Cont{\infty}_c(\Omega)$. If $\xi=(\zeta,\xi_0)$, take an orthonormal basis $\{\zeta_i\}$ of $\R^d$ with $\zeta_1=\zeta$ and $O=(\zeta_1|\cdots|\zeta_d)\in\textrm{SO}(d)$. We make first the change of variables $x=Ox'=\sum_{i=1}^dx_i'\zeta_i$ with $\Omega'=O^{\mathrm{T}}\Omega$ and $G(x')=g(Ox')$
\begin{align*}\int_{\Omega}g(x)A(h(k\xi\cdot(x,t)))\dif x
&=\int_{\Omega'}G(x')A(h(kx_1'+k\xi_0t))\dif x',\\
\int_{\Omega}g(x)\dif x
&=\int_{\Omega'}G(x')\dif x'.
\end{align*}
After that, we integrate by parts
\begin{align*}
\int_{\Omega'}G(x')A(h(kx_1'+k\xi_0t))\dif x'&=
-\frac{1}{k}\int_{\Omega'}\partial_1G(x')\int_{0}^{kx_1'}A(h(\tau+k\xi_0t))\dif\tau\dif x',\\
\int_{\Omega'}G(x')\dif x'
&=-\int_{\Omega'}\partial_1G(x')x_1'\dif x'.
\end{align*}
Finally, by adding and subtracting
$$\frac{1}{k}\int_{\Omega'}\partial_1G(x')\int_{0}^{\lceil kx_1'\rceil}A(h(\tau+k\xi_0t))\dif\tau\dif x'
=\frac{1}{k}\int_{\Omega'}\partial_1G(x')\lceil kx_1'\rceil\int_{0}^{1}A(h(\tau))\dif\tau\dif x'
$$
where $\lceil\cdot\rceil$ is the ceiling function, we get
\begin{align*}&\left|\int_{\Omega}g(x)A(h(k\xi\cdot(x,t)))\dif x
-\int_{\Omega}g(x)\dif x\int_{0}^{1}A(h(\tau))\dif\tau\right|\\
&=\frac{1}{k}\left|\int_{\Omega'}\partial_1G(x')\left[\int_{0}^{\lceil kx_1'\rceil-kx_1'}A(h(\tau+k\xi_0t))\dif\tau +(kx_1'-\lceil kx_1'\rceil)\int_{0}^{1}A(h(\tau))\dif\tau\right]\dif x'\right|\\
&\leq\frac{2}{k}\norma{\partial_1G}{\Cont{0}(\Omega')}\norma{A}{\Cont{0}(B_h)}
\end{align*}
being $B_h$ the ball of radius $\norma{h}{\Leb{\infty}(\T)}$. Therefore, \eqref{lemmaYM1} follows.
By density, the result is extended for all $g\in\Leb{1}(\Omega)$.
\end{proof}

The following lemma shows that the $\Lambda$-segments in  \textrm{(H2)} can be selected uniformly away from the boundary on compact sets.
\begin{lema}\label{lemasegmentos} Let $\Lambda$ from \textrm{(H1)} and $\Dfunc,\Hfunc,\Phi$ from \textrm{(H2)}. For every $z\in U$, define
$$\Sigma_z=\{\sigma=[-\bar{z},\bar{z}]\subset\Lambda\,:\,\Hfunc(\bar{z})\geq\Phi(\Dfunc(z)),\, z+\sigma\subset U\}$$
which is non-empty by \textrm{(H2)}. Then, the function
\function{D}{U}{(0,\infty)}{z}{\sup_{\sigma\in\Sigma_z}\dist{z+\sigma}{\partial U}}
is lower-semicontinuous.
\end{lema}
\begin{proof} Fix $z_0\in U$ and $0<\varepsilon<\tfrac{1}{2}$. 
The definition of $D$ yields a $\sigma_0=[-\bar{z}_0,\bar{z}_0]\in\Sigma_{z_0}$ so that
$$\dist{z_0+\sigma_0}{\partial U}\geq(1-\varepsilon)D(z_0)>0.$$
Now let $(z_k)\subset U$ with $z_k\rightarrow z_0$. Since $\Phi\circ\Dfunc$ is continuous and positive on $U$, the term
\begin{equation}\label{lemasegmentos1}
\delta_k=\left|\left(\frac{\Phi(\Dfunc(z_k))}{\Phi(\Dfunc(z_0))}\right)^{\frac{1}{\gamma}}-1\right|\rightarrow 0
\end{equation}
where $\gamma$ is the degree of homogeneity of $\Hfunc$ (def. \ref{semistronglyconcave}).
We take $\lambda_k=1+\delta_k$ and $\sigma_k=\lambda_k\sigma_0=[-\lambda_k\bar{z}_0,\lambda_k\bar{z}_0]$. Let us show that, for a big enough $k_0$, we have $\sigma_k\in\Sigma_{z_k}$ for all $k\geq k_0$. On the one hand, since $\Hfunc$ is positive homogeneous of degree $\gamma$, $\sigma_0\in\Sigma_{z_0}$ and \eqref{lemasegmentos1}, we have
$$\Hfunc(\lambda_k\bar{z}_0)=\lambda_k^\gamma\Hfunc(\bar{z}_0)\geq\lambda_k^\gamma\Phi(\Dfunc(z_0))\geq\Phi(\Dfunc(z_k)).$$ 
On the other hand, by adding and subtracting  
$z_0+\lambda\bar{z}_0$, the triangle inequality implies
\begin{align*}
\dist{z_k+\sigma_k}{\partial U}&=\min_{\substack{|\lambda|\leq 1 \\ z'\in\partial U}}|(z_k+\lambda\lambda_k\bar{z}_0)-z'|\\
&\geq\dist{z_0+\sigma_0}{\partial U}-|z_k-z_0|-\delta_k|\bar{z}_0|.
\end{align*}
Hence, for a big enough $k_0$, $\dist{z_k+\sigma_k}{\partial U}\geq(1-2\varepsilon)D(z_0)$ and consequently $z_k+\sigma_k\subset U$ for all $k\geq k_0$. Therefore
\begin{equation}\label{lemasegmentos2}
D(z_k)\geq\dist{z_k+\sigma_k}{\partial U}\geq(1-2\varepsilon)D(z_0)
\end{equation}
for all $k\geq k_0$. Finally, by computing the $\liminf$ on \eqref{lemasegmentos2} and then making $\varepsilon\downarrow 0$ we deduce that $D$ is lower-semicontinuous at $z_0$.
\end{proof}

The key point to prove the quantitative h-principle is the following perturbation property. The steps 1, 2 and 4 in the proof are an adaptation of the proof of \cite[prop. 3]{Onadmissibility}. We recall it for convenience. The step 3 is the new requirement from (\textrm{H3},$\Fspace_0$) and our main contribution in this scheme. More precisely, although the approximating sequence is constructed in the same way as in \cite[prop. 3]{Onadmissibility}, we need to check that it belongs to our $\Sub$, i.e., that it satisfies (\textrm{H3},$\Fspace_0$).

\begin{prop}[Perturbation property]\label{perturbationproperty} For every $\mu>0$, there exists $\beta(\Omega,\mu)>0$ such that, for all $z\in\Sub$ satisfying
$$\mathcal{J}(z)\geq\mu,$$
there exists a sequence $(z_k)\subset\Sub$ with $z_k\overset{d}{\rightarrow}z$ so that
\begin{equation}\label{betaprop}
\mathcal{J}(z)\geq\limsup_{k\rightarrow\infty}\mathcal{J}(z_k)+\beta.
\end{equation}
\end{prop}
\begin{proof} \textbf{Step 1. The shifted grid and the discretization.} First we recall how the shifted grid and the discretization are constructed in \cite[prop. 3 step 1]{Onadmissibility}. Let $s>0$ be
the side length of the regular cubes in the grid to be determined.
Denote $I^s=[t_1-s,t_2+s]$ and $I_s=I^s\cap[0,T]=[t_1-s,(t_2+s)\wedge T]$. Fix $0<s_0\leq\frac{1}{2}t_1$ such that $\Omega\times I_{s_0}\Subset\DomainP$.  
For $\zeta\in\Z^d$, let $Q_\zeta^s$ and $\tilde{Q}_\zeta^s$ be the regular cubes in $\R^d$ centered at $s\zeta$ with side length $s$ and $\frac{3}{4}s$ respectively. Next, for $(\zeta,i)\in\Z^d\times\Z$, depending on $\sum_j\zeta_j\in 2\Z+b$ for some binary number $b\in\{0,1\}$, define
$$C_{\zeta,i}^s=
Q_\zeta^s\times I_{i,b}^s,\quad
\tilde{C}_{\zeta,i}^s=
\tilde{Q}_\zeta^s\times\tilde{I}_{i,b}^s,$$
where $I_{i,b}^s$ and $\tilde{I}_{i,b}^s$ are the intervals in $\R$ centered at $s(i+\frac{b}{2})$ with length $s$ and $\frac{3}{4}s$ respectively
(see \cite[fig. 1]{Onadmissibility}). For each $b\in\{0,1\}$, define
$$\Omega_b^s=
\bigcup_{\substack{C_{\zeta,i}^s\subset\Omega\times I^s \\ {|\zeta|\in 2\Z+b}}}\tilde{Q}_\zeta^s,
\quad
I_b^s=\bigcup_{I_{i,b}^s\subset I^s}\tilde{I}_{i,b}^s.$$
Take 
$\chi\in\Cont{\infty}_c((0,1);[0,1])$ with $\chi|_{(\frac{1}{8},\frac{7}{8})}=1$. Consider
$\psi_{\zeta,i}^s(x,t)=\chi_{\zeta_1}^s(x_1)\cdots\chi_{\zeta_d}^s(x_d)\chi_{i,b}^s(t)=\chi_{\zeta}^s(x)\chi_{i,b}^s(t)$
the corresponding cut-off function on $C_{\zeta,i}^s$ with $\psi_{\zeta,i}^s|_{\tilde{C}_{\zeta,i}^s}=1$, i.e., $\chi_{\zeta_j}^s=\chi\left(\frac{\cdot-s\zeta_j}{s}\right)$ and
$\chi_{i,b}^s=\chi\left(\frac{\cdot-s(i+b/2)}{s}\right)$.
Now, for every $f:\R^{d+1}\rightarrow\R$ we define its discretization in the grid as the simple function
$$\boxdot_sf=\sum_{(\zeta,i)\in\Z^d\times\Z}f(\expected{\tilde{C}_{\zeta,i}^s})\car{\tilde{C}_{\zeta,i}^s}.$$
For uniformly continuous functions $f$ on $\Omega\times I_{s_0}$ it follows that, for any  $b\in\{0,1\}$
$$\sup_{t\in I}\left|\int_{\Omega_b^s}\boxdot_sf(x,t)\dif x-\tfrac{1}{2}\left(\tfrac{3}{4}\right)^d\int_{\Omega}f(x,t)\dif x\right|\rightarrow 0$$
when $s\downarrow 0$. Hence, since $\Dfunc\circ z$ and the constant function $1$ are uniformly continuous on $\Omega\times I_{s_0}$, there exists $0<s_1\leq s_0$ depending on $\Dfunc$, $z$, $\mu$ and $|\Omega|$ so that 
\begin{align}\label{perturbationproperty1}
\sup_{t\in I}\left|\int_{\Omega_b^s}\boxdot_s\Dfunc(z(x,t))\dif x-\tfrac{1}{2}\left(\tfrac{3}{4}\right)^d\int_{\Omega}\Dfunc(z(x,t))\dif x\right|&\leq\tfrac{1}{8}\left(\tfrac{3}{4}\right)^d\mu,\\
\left||\Omega_b^s|-\tfrac{1}{2}\left(\tfrac{3}{4}\right)^d|\Omega|\right|&\leq\tfrac{1}{4}\left(\tfrac{3}{4}\right)^d|\Omega|,
\end{align}
for every $b\in\{0,1\}$ and $0<s\leq s_1$.\\
Now, we need to refine the grid to guarantee that the $\Lambda$-segment from lemma \ref{lemasegmentos} associated to the image of the middle point of each cube in the grid is also away from the boundary for the rest of the points of  the cube.
Since $z(\Omega\times I_{s_0})\Subset U$, lemma \ref{lemasegmentos} ensures that there exists $\delta(z,\Omega,I)>0$ so that, for all $y\in\Omega\times I_{s_0}$ there is a $\Lambda$-segment $\sigma_{z(y)}=[-\bar{z}(y),\bar{z}(y)]\in\Sigma_{z(y)}$  satisfying
$$\dist{z(y)+\sigma_{z(y)}}{\partial U}\geq\delta.$$
Let us fix $0<s_2\leq s_1$ such that
$|z(y)-z(y')|\leq\tfrac{1}{2}\delta$
whenever $\max_j|y_j-y_j'|\leq \tfrac{1}{2}s_2$. From now on, whenever there is no ambiguity, we skip $s_2$ to simplify the notation.\\
\indent \textbf{Step 2. The perturbation.} Here we recall how the perturbation is constructed in \cite[prop. 3 step 2]{Onadmissibility}.
For each $C_{\zeta,i}\subset\Omega\times I^{s_2}$, denote $y_{\zeta,i}=\expected{\tilde{C}_{\zeta,i}}$, $z_{\zeta,i}=z(y_{\zeta,i})$ and  $\sigma_{\zeta,i}=\sigma_{z_{\zeta,i}}$. Let $\xi_{\zeta,i}\in\Sp^{d-1}\times\R$ be the direction and $\tilde{z}_{\zeta,i}^k$ the localized smooth solution in \textrm{(H1)} associated to $\bar{z}_{\zeta,i}$ and $\psi_{\zeta,i}$.
Then, since
$$z(y)+\tilde{z}_{\zeta,i}^k(y)
=(z_{\zeta,i}+\bar{z}_{\zeta,i}h(k\xi_{\zeta,i}\cdot y)\psi_{\zeta,i}(y))+(z(y)-z(y_{\zeta,i}))+\mathcal{O}(k^{-1})$$
for all $y\in C_{\zeta,i}$, for a big enough $k_0$
$$z(y)+\tilde{z}_{\zeta,i}^k(y)\in 
B_{\frac{\delta}{2}+\mathcal{O}(k^{-1})}(z_{\zeta,i}+\sigma_{\zeta,i})\in U$$
for all $y\in C_{\zeta,i}$ and $k\geq k_0$. We define the perturbation as usual
$$\tilde{z}_k=\sum_{C_{\zeta,i}\subset\Omega\times I^{s_2}}\tilde{z}_{\zeta,i}^k\quad\textrm{and}\quad z_k=z+\tilde{z}_k.$$
Hence, for every $\varphi\in\Leb{\qq}(\Domain)$, by applying \textrm{(H1)} and lemma \ref{lemmaYM}, we get
\begin{align*}&\left|\int_{\Domain}\varphi(x)\cdot(z_k(x,t)-z(x,t))\dif x\right|
=\sum_{C_{\zeta,i}\subset\Omega\times I^{s_2}}\left|\int_{Q_\zeta}\varphi(x)\cdot \tilde{z}_{\zeta,i}^k(x,t)\dif x\right|
\\
&\leq\sum_{C_{\zeta,i}\subset\Omega\times I^{s_2}}\Bigg(\norma{\varphi}{\Leb{\qq}(\Domain)}\norma{\tilde{z}_{\zeta,i}^k(\cdot,t)-\bar{z}_{\zeta,i}h(k\xi_{\zeta,i}\cdot (\cdot,t))\psi_{\zeta,i}(\cdot,t)}{\Leb{\pp}(Q_{\zeta})}
\\
&\hspace{2.5cm}
+|\bar{z}_{\zeta,i}|\left|\int_{Q_\zeta}\varphi(x)\cdot\chi_{\zeta}(x)h(k\xi_{\zeta,i}\cdot(x,t))\dif x\right|\Bigg)\\
&\rightarrow \sum_{C_{\zeta,i}\subset\Omega\times I^{s_2}}|\bar{z}_{\zeta,i}|\left|\int_{Q_\zeta}\varphi(x)\cdot\chi_{\zeta}(x)\dif x\right|\left|\int_{\T} h\dif\tau\right|=0
\end{align*}
uniformly in $t\in[0,T]$ when $k\rightarrow\infty$ (notice the sum is finite). Therefore, 
$z_k\rightarrow z$
in $\Cont{0}([0,T];\Leb{\pp}_{w^*}(\Domain))$ when $k\rightarrow\infty$, in particular in $\Cont{0}([0,T];\Leb{\pp}_{\LL}(\Domain;\tilde{\Kconstrain}))$.\\
\indent \textbf{Step 3. The $(\Fspace_0)$-property.} Here we want to show that, for a big enough $k_1\geq k_0$, $z_k\in\Sub$ for all $k\geq k_1$. Since we have checked \textrm{(H3,$U$)}, it is enough to show \textrm{(H3,$\Fspace_0$)}. The idea of the proof is that, for small cubes the property holds immediately, whereas for large cubes one can reduce to a finite number of cubes and times in order to exploit then the convergence in $\Cont{0}([0,T];\Leb{\pp}_{\LL}(\Domain;\tilde{\Kconstrain}))$.
Fix $0<C'(z)<1-C(z)$ where $C(z)\in(0,1)$ is the constant of $z$ from \textrm{(H3,$\Fspace_0$)}. Since $\Fspace_0$ is finite, without loss of generality we may assume $\Fspace_0=\{(\FF,g,\mapy)\}$ for simplicity. If $\tilde{\mapy}:\DomainPy\rightarrow\DomainP$ denotes the homeomorphism defined by $\tilde{\mapy}(x,t)=(\mapy(x,t),t)$, then
$\tilde{\mapy}^{-1}(\Omega\times I_{s_0})\Subset\DomainPy$. On the one hand, the projection of $\tilde{\mapy}^{-1}(\Omega\times I_{s_0})$ into $\R^d$, $\bigcup_{t\in I_{s_0}}\mapy^{-1}(\Omega,t)$,
is bounded so we can take a bounded regular cube $Q_0\subset\R^d$ containing it. On the other hand, we define the constants
$$\epsilon=\epsilon(\Omega,I,\Fspace_0)=1\wedge\inf_{t\in I_{s_0}}\dist{\mapy^{-1}(\Omega,t)}{\partial\DomainPy(t)}>0$$ 
and
$$E=E(\Omega,I,\Fspace_0)=\min\left\lbrace\Efunc(\varsigma,t)\,:\, t\in I_{s_0},\,\varsigma\in\DomainPy(t)\,
\textrm{ s.t. }
\,\dist{\varsigma}{\partial\DomainPy(t)}\geq\tfrac{\epsilon}{2(\sqrt{d}+1)}\right\rbrace>0.$$
Hence, for every cube $Q\subset\DomainPy(t)$ at each $t\in I_{s_2}$ with $\mapy(Q,t)\cap\Omega\neq\emptyset$, necessarily
$$\Efunc(\expected{Q},t)\geq E.$$
Otherwise, it would be $\dist{\expected{Q}}{\partial\DomainPy(t)}<\tfrac{\epsilon}{2(\sqrt{d}+1)}$ and consequently
$$
\dist{\varsigma}{\partial\DomainPy(t)}\leq
\dist{\varsigma}{\expected{Q}}+\dist{\expected{Q}}{\partial\DomainPy(t)}
\leq(\sqrt{d}+1)\dist{\expected{Q}}{\partial\DomainPy(t)}<\tfrac{1}{2}\epsilon
$$
for all $\varsigma\in Q$, so $Q\cap\mapy^{-1}(\Omega,t)=\emptyset$.
Now, since $z\in\Sub$, for every cube $Q\subset\DomainPy(t)$ at each $t\in (0,T]$,
\begin{align*}
&\left|\int_{\mapy(Q,t)}[\FF(z_k)-\FF(\breve{z})](x,t)g(x,t)\dif x\right|\\
&\leq C(z)\Efunc(\expected{Q},t)(1\wedge|\mapy(Q,t)|^\alpha)+\left|\int_{\mapy(Q',t)}[\FF(z_k)-\FF(z)](x,t)g(x,t)\dif x\right|
\end{align*}
where $Q'=Q\cap Q_0$ because there is not perturbation outside $Q_0$.
Indeed, since there is not perturbation outside $\Omega\times I_{s_2}$, it is enough to show that there is a big enough $k_1\geq k_0$ such that
$$\left|\int_{\mapy(Q',t)}[\FF(z_k)-\FF(z)](x,t)g(x,t)\dif x\right|\leq C'E(1\wedge|\mapy(Q',t)|^\alpha)$$
for every cube $Q'\subset Q_0\cap\DomainPy(t)$ at each $t\in I_{s_2}$ and $k\geq k_1$.\\ 
Now, since $z_k([0,T]),z([0,T])\subset\tilde{\BB}\Subset\Leb{\pp}_{\LL}(\Omega;\tilde{\Kconstrain})$ (recall \textrm{(H3)}) and $\FF\in\Cont{}(\Leb{\pp}_{\LL}(\Omega;\tilde{\Kconstrain});\Leb{p}_{w^*}(\Omega))$,
then $\FF(z_k([0,T])),\FF(z([0,T]))\subset\FF(\tilde{\BB})\Subset\Leb{p}_{w^*}(\Omega)$. 
Consider the constant
$$B=B(\Omega,I,\Fspace_0)=\sup_{z\in\tilde{\BB}}\norma{\FF(z)}{\Leb{p}(\Omega)}\norma{g}{\Cont{0}(I_{s_0};\Leb{q}(\Omega))}
<\infty.$$
Then, for every $Q'$ such that $|\mapy(Q',t)|\leq 1$ and
$2B|\mapy(Q',t)|^{\frac{1}{r}-\alpha}\leq
C'E$, H\"{o}lder inequality implies
\begin{align*}
\int_{\mapy(Q',t)}|[\FF(z_k)-\FF(z)](x,t)||g(x,t)|\dif x\leq 2B
|\mapy(Q',t)|^{\frac{1}{r}}\leq C'E(1\wedge|\mapy(Q',t)|^\alpha).
\end{align*}
For all the rest $Q'$, there is a constant $D(\Omega,I,\Fspace_0,C')>0$ such that
\begin{equation}\label{Dconstant}
(1\wedge|\mapy(Q',t)|^\alpha)\geq D.
\end{equation}
Let us fix a fine enough and finite families of cubes and times respectively.
For the cubes, let $j\in\N$ such that
\begin{equation}\label{Fprop1}2B\sup_{t\in I_{s_0}}\norma{\Jacobian_{\mapy(t)}}{\Cont{0}(\mapy^{-1}(\Omega,t))}^{\frac{1}{r}}(2^{1-j}d|Q_0|)^{\frac{1}{r}}
\leq\tfrac{1}{4}C'ED
\end{equation}
where $\Jacobian_{\mapy(t)}$ is the Jacobian of $\mapy(t)$. With this $j$ we construct the homogeneous grid in $Q_0$ with side length $2^{-j}\ell(Q_0)$, being $\ell(Q_0)$ the side length of $Q_0$.
Then, if 
$\{Q_{1}',\ldots Q_{a_0}'\}$ is the finite family of all possible cubes in the grid, for every cube $Q'\subset Q_0$
there is a maximal cube $Q_{a}'\subset Q'$ ($a\in\{1,\ldots,a_0\}$) so that
\begin{equation}\label{Fprop2}
|Q'\setminus Q_{a}'|\leq\underbrace{2d}_{\textrm{faces}}\underbrace{(2^j)^{d-1}}_{cubes}\underbrace{(2^{-j}\ell(Q_0))^d}_{volume}
=2^{1-j}d|Q_0|.
\end{equation}
In particular, for every cube $Q'\subset Q_0\cap\DomainPy(t)$ at each $t\in I_{s_2}$, \eqref{Fprop1}, \eqref{Fprop2} and H\"{o}lder inequality implies
\begin{align}\label{Fprop3}\int_{\mapy(Q'\setminus Q_a',t)}|[\FF(z_k)-\FF(z)](x,t)||g(x,t)|\dif x
&\leq 2B\sup_{t\in I_{s_2}}\norma{\Jacobian_{\mapy(t)}}{\Cont{0}(\mapy^{-1}(\Omega,t))}^{\frac{1}{r}}|Q'\setminus Q_a'|^{\frac{1}{r}}\nonumber\\
&\leq\tfrac{1}{4}C'ED.\end{align}
For the times, since 
$$t\mapsto G_{a}(\cdot,t)=g(\cdot,t)\car{\mapy(Q_a'\cap\mapy^{-1}(\Omega,t),t)}(\cdot)$$
is uniformly continuous from $I_{s_2}$ to $\Leb{p^*}(\Omega)$,
we can take a finite family of times $\{t_1,\ldots,t_{c_0}\}\subset I_{s_2}$ such that, for every $t\in I_{s_2}$ there is $c\in\{1,\ldots, c_0\}$ so that
\begin{equation}\label{Fprop4}
2\sup_{z\in\tilde{\BB}}\norma{\FF(z)}{\Leb{p}(\Omega)}\norma{G_{a}(\cdot,t)-G_{a}(\cdot,t_c)}{\Leb{p^*}(\Omega)}
\leq\tfrac{1}{4}C'ED
\end{equation}
for all $a\in\{1,\ldots,a_0\}$. 
Once we have chosen these families, since $z_k\rightarrow z$
in $\Cont{0}([0,T];\Leb{\pp}_{\LL}(\Domain;\tilde{\Kconstrain}))$ and $\FF\in\Cont{}(\Leb{\pp}_{\LL}(\Omega;\tilde{\Kconstrain});\Leb{p}_{w^*}(\Omega))$,
we can take a big enough $k_1\geq k_0$ such that
\begin{equation}\label{Fprop5}
\sup_{t\in I_{s_2}}\left|\int_{\Omega}[\FF(z_k)-\FF(z)](x,t)G_{a}(x,t_c)\dif x\right|\leq\tfrac{1}{2}C' ED
\end{equation}
for all $a\in\{1,\ldots a_0\}$, $c\in\{1,\ldots c_0\}$ and $k\geq k_1$.
Finally, for any $Q'\subset Q_0\cap\DomainPy(t)$ and $t\in I_{s_2}$ satisfying \eqref{Dconstant} consider $Q_a'$ and $t_c$ as before. Then, 
by adding and subtracting
\begin{align*}
\int_{\mapy(Q_a',t)}[\FF(z_k)-\FF(z)](x,t)g(x,t)\dif x
&=\int_{\mapy(Q_a'\cap\mapy^{-1}(\Omega,t),t)}[\FF(z_k)-\FF(z)](x,t)g(x,t)\dif x\\
&=\int_{\Omega}[\FF(z_k)-\FF(z)](x,t)G_a(x,t)\dif x
\end{align*}
where we have used that $z_k=z$ outside $\Omega$, and
$$\int_{\Omega}[\FF(z_k)-\FF(z)](x,t)G_a(x,t_c)\dif x,$$
\eqref{Fprop3}-\eqref{Fprop5} yields
\begin{align*}
&\left|\int_{\mapy(Q',t)}[\FF(z_k)-\FF(z)](x,t)g(x,t)\dif x\right|\\
&\leq 2B\sup_{t\in I_{s_2}}\norma{\Jacobian_{\mapy(t)}}{\Cont{0}(\mapy^{-1}(\Omega,t))}^{\frac{1}{r}}|Q'\setminus Q_a'|^{\frac{1}{r}}+
2\sup_{z\in\tilde{\BB}}\norma{\FF(z)}{\Leb{p}(\Omega)}\norma{G_{a}(\cdot,t)-G_{a}(\cdot,t_c)}{\Leb{p^*}(\Omega)}\\
&\quad+
\left|\int_{\Omega}[\FF(z_k)-\FF(z)](x,t)G_a(x,t_c)\dif x\right|\leq C' ED
\end{align*}
for all $k\geq k_1$.\\
\indent \textbf{Step 4. The $\beta$-property.} Here we follow \cite[prop. 3 step 3]{Onadmissibility} but replacing ``$e-\tfrac{1}{2}|v|^2$'' by $\Dfunc$. Let $b\in\{0,1\}$. Then, for every $C_{\zeta,i}\subset\Omega_b$,	$$\Hfunc(\tilde{z}_k(y))=\Hfunc(\bar{z}_{\zeta,i}h(k\xi_{\zeta,i}\cdot y)),\quad y\in \tilde{C}_{\zeta,i},$$
where $\Hfunc$ is given in def. \ref{semistronglyconcave}. Hence, lemma \ref{lemmaYM} implies
$$\lim_{k\rightarrow\infty}
\int_{\tilde{Q}_\zeta}\Hfunc(\tilde{z}_k(x,t))\dif x=
\int_{\tilde{Q}_\zeta}\int_{\T}\Hfunc(\bar{z}_{\zeta,i}h(\tau))\dif\tau\dif x
=C_\gamma\Hfunc(\bar{z}_{\zeta,i})|\tilde{Q}_\zeta|$$
uniformly in $t\in I_b^{s_2}\cap I$, where $C_\gamma=\norma{h}{\Leb{\gamma}(\T)}^\gamma>0$. Let $\Phi^*$ be the convex-envelope (see \cite[def. 1.7]{Rigidity}) of $\Phi$ in \textrm{(H2)}, which is also increasing with $\Phi^*\in\Cont{}((0,1];(0,\infty))$ and $\Phi^*\leq\Phi$. Then, 
$$\Hfunc(\bar{z}_{\zeta,i})\geq\Phi^*(\Dfunc(z_{\zeta,i}))
=\Phi^*(\boxdot\Dfunc(z(y))),
\quad y\in\tilde{C}_{\zeta,i}.$$
Therefore, for each $b\in\{0,1\}$, by summing over all the cubes and applying Jensen inequality, we get
\begin{align*}&\lim_{k\rightarrow\infty}\int_{\Omega_b}\Hfunc(\tilde{z}_k(x,t))\dif x
\geq C_\gamma\int_{\Omega_b}\Phi^*(\boxdot\Dfunc(z(x,t)))\dif x\\
&\geq C_\gamma|\Omega_b|\Phi^*\left(\frac{1}{|\Omega_b|}\int_{\Omega_b}\boxdot\Dfunc(z(x,t))\dif x\right)
\geq \tfrac{C_\gamma}{4}\left(\tfrac{3}{4}\right)^d|\Omega|\Phi^*\left(\frac{1}{|\Omega|}\int_{\Omega_b}\boxdot\Dfunc(z(x,t))\dif x\right)
\end{align*}
uniformly in $t\in I_b^{s_2}\cap I$. In general,
\begin{align*}\liminf_{k\rightarrow\infty}\int_{\Omega}\Hfunc(\tilde{z}_k(x,t))\dif x
\geq \tfrac{C_\gamma}{4}\left(\tfrac{3}{4}\right)^d|\Omega|\Phi^*\left(\frac{1}{|\Omega|}\min_{b\in\{0,1\}}\int_{\Omega_b}\boxdot\Dfunc(z(x,t))\dif x\right)
\end{align*}
uniformly in $t\in(I_0^{s_2}\cup I_1^{s_2})\cap I=I$.
Since $z_k=z+\tilde{z}_k$, the semistrongly concavity of $\Dfunc$ implies
\begin{align*}
&\int_{\Omega}\Dfunc(z_k(x,t))\dif x
\leq\int_{\Omega}\Dfunc(z(x,t))\dif x
+\int_{\Omega}\Gfunc(z(x,t))\cdot\tilde{z}_k(x,t)\dif x
-\int_{\Omega}\Hfunc(\tilde{z}_k(x,t))\dif x.
\end{align*}
Since $\Gfunc\circ z$ is uniformly continuous on $\Omega\times I_{s_0}$, without loss of generality (by taking a subsequence and relabelling if necessary), we may assume the linear term goes to zero uniformly in $t\in I$ when $k\rightarrow\infty$. 
Then,
\begin{align*}
&\limsup_{k\rightarrow\infty}\int_{\Omega}\Dfunc(z_k(x,t))\dif x
\leq\int_{\Omega}\Dfunc(z(x,t))\dif x-\tfrac{C_\gamma}{4}\left(\tfrac{3}{4}\right)^d|\Omega|\Phi^*\left(\frac{1}{|\Omega|}\min_{b\in\{0,1\}}\int_{\Omega_b}\boxdot\Dfunc(z(x,t))\dif x\right)
\end{align*}
uniformly in $t\in I$. Finally, at each $t\in I$, if
$$\int_{\Omega}\Dfunc(z(x,t))\dif x\leq\tfrac{1}{2}\mu,
$$
then directly
$$\limsup_{k\rightarrow\infty}\int_{\Omega}\Dfunc(z_k(x,t))\dif x\leq\tfrac{1}{2}\mu\leq\mathcal{J}(z)-\tfrac{1}{2}\mu.$$
Otherwise, by applying \eqref{perturbationproperty1}, 
$$\limsup_{k\rightarrow\infty}\int_{\Omega}\Dfunc(z_k(x,t))\dif x\leq
\mathcal{J}(z)-\tfrac{C_\gamma}{4}\left(\tfrac{3}{4}\right)^d|\Omega|\Phi^*\left(\tfrac{1}{8}\left(\tfrac{3}{4}\right)^d\frac{\mu}{|\Omega|}\right).
$$
In general,
\eqref{betaprop} holds for
$$\beta(\Omega,\mu)=\min\left\{\tfrac{1}{2}\mu, \tfrac{C_\gamma}{4}\left(\tfrac{3}{4}\right)^d|\Omega|\Phi^*\left(\tfrac{1}{8}\left(\tfrac{3}{4}\right)^d\frac{\mu}{|\Omega|}\right)\right\}.$$
This concludes the proof.
\end{proof}

As a consequence of the perturbation property, we deduce the quantitative h-principle that we are looking for. In addition, we show a corollary which can be though as a generalization of the ``mix in space at each time slice'' property (def. \ref{degradada}\ref{DMS2}) for the Muskat-Mixing problem.\newpage

\begin{thm}[Quantitative h-principle]\label{HP} The set of functions $z\in\SSub$ satisfying:
\begin{enumerate}[(i)]
\item\label{HP1} $z$ is an exact solution.
\item\label{HP2} For all $(\FF,g,\mapy)\in\Fspace_0$, at each $t\in(0,T]$
$$
\left|\dashint_{\mapy(Q,t)}\left[\FF(z)-\FF(\breve{z})\right](x,t)g(x,t)\dif x\right|\leq\Efunc(\expected{Q},t)\Qfunc{\mapy(Q,t)}
$$
for every non-empty (non-necessarily regular) bounded cube $Q\subset\DomainPy(t)$.
\end{enumerate}
contains a residual set in $\SSub$. 
\end{thm}
\begin{proof} Let us start showing that all $z\in\SSub$ satisfies \textit{\ref{HP2}}. Take $(z_k)\subset\Sub$ with 
$z_k\overset{d}{\rightarrow}z$. Fix $t\in(0,T]$ and  $Q\subset\DomainPy(t)$. Hence, since $(\FF,g,\mapy)\in\Fspace_0$,
we have
$$\int_{\mapy(Q,t)}\FF(z_k)(x,t)g(x,t)\dif x\rightarrow\int_{\mapy(Q,t)}\FF(z)(x,t)g(x,t)\dif x$$
when $k\rightarrow\infty$. This implies \textit{\ref{HP2}}. 
Let us show now that the set of exact solutions contains a residual set. For all $(x_0,t_0)\in\DomainP$ let $(x_0,t_0)\in\Omega\times I\Subset\DomainP$ and the associated relaxation-error functional $\mathcal{J}$.	
Following \cite{Onadmissibility}, the perturbation property implies $\SSub_{\mathcal{J}}\subset\mathcal{J}^{-1}(0)$.  
Thus, 
by covering $\DomainP$ (second countable) with a countable family $\{\Omega_j\times I_j\}_{j\in\N}$, we deduce that
$\bigcap_{j\in\N}\SSub_{\mathcal{J}_{j}}$
is a residual set in $\SSub$ contained in
$\bigcap_{j\in\N}\mathcal{J}_{j}^{-1}(0)\subset\{z\in\SSub\,:\, z\textrm{ exact solution}\}$ (recall \eqref{J0}).
\end{proof}

\begin{cor}\label{corHP} Suppose that for some $(\FF,g,\mapy)\in\Fspace_0$ there is a compact set $C\Subset\R$ so that, for every  $z\in\Sub$:
\begin{itemize}[(iii)]
	\item\label{HPMixprop} At each $t\in(0,T]$, for every cube $Q\Subset\DomainPy(t)$ with rational extreme points
	$$\int_{\mapy(Q,t)}\FF(z)(x,t)g(x,t)\dif x\notin C.$$
\end{itemize}
Then, the set of functions $z\in\SSub$ satisfying \ref{HP1}, \ref{HP2} and \textit{(iii)} contains a residual set in $\SSub$.
\end{cor}
\begin{proof} Let $Q$ be a cube and $I$ a time interval with rational extreme points and $Q\times I\Subset\DomainPy$. Define
$$C_{Q,I}=\left\{z\in\SSub\,:\,\int_{\mapy(Q,t)}\FF(z)(x,t)g(x,t)\dif x\in C\quad\textrm{for some }t\in I\right\}.$$
We claim that $C_{Q,I}$ is closed in $\SSub$. Let $(z_k)\subset C_{Q,I}$ with $z_k\rightarrow z$ in $\SSub$. By definition, for each $k$ there is $t_k\in I$ and $c_k\in C$ such that
$$\int_{\mapy(Q,t_k)}\FF(z)(x,t_k)g(x,t_k)\dif x=c_k.$$
Since $I$ and $C$ are compact, without loss of generality we may assume $t_k\rightarrow t_0$ and $c_k\rightarrow c_0$ for some $t_0\in I$ and $c_0\in C$ respectively. On the one hand, since
$\Omega=\bigcup_{t\in I}\mapy(Q,t)$
is bounded in $\Domain$ and
$$t\mapsto G(\cdot,t)=g(\cdot,t)\car{\mapy(Q,t)}(\cdot)$$
is continuous from $I$ to $\Leb{p^*}(\Omega)$, then
\begin{align*}&\int_{\Omega}|\FF(z_k)(x,t_k)||G(x,t_k)-G(x,t_0)|\dif x	\leq\sup_{z\in\tilde{\BB}}\norma{\FF(z)}{\Leb{p}(\Omega)}\norma{G(\cdot,t_k)-G(\cdot,t_0)}{\Leb{p^*}(\Omega)}\rightarrow 0
\end{align*}
when $k\rightarrow\infty$. On the other hand, since
\begin{align*}
d_{\BB}(z_k(t_k),z(t_0))
&\leq d(z_k,z)+d_{\BB}(z(t_k),z(t_0)) \rightarrow 0
\end{align*}
when $k\rightarrow\infty$, we have $\FF(z_k(t_k))\rightarrow\FF(z(t_0))$
in $\Leb{p}_{w^*}(\Domain)$ and
$$\lim_{k\rightarrow\infty}\int_{\Omega}\FF(z_k)(x,t_k)G(x,t_0)\dif x
=\int_{\Omega}\FF(z)(x,t_0)G(x,t_0)\dif x.$$
Therefore,
$$\int_{\mapy(Q,t_0)}\FF(z)(x,t_0)g(x,t_0)\dif x
=\lim_{k\rightarrow\infty}\int_{\mapy(Q,t_k)}\FF(z_k)(x,t_k)g(x,t_k)\dif x
=\lim_{k\rightarrow\infty}c_k=c_0,$$
and $z\in C_{Q,I}$.
Since $C_{Q,I}\cap\Sub=\emptyset$ and $C_{Q,I}\subset\SSub=\overline{\Sub}$, necessarily $(C_{Q,I})^\circ=\emptyset$, so $\SSub\setminus(C_{Q,I})$
is open and dense in $\SSub$. Since they are countable, the intersection is a residual set in $\SSub$. 
\end{proof}

\begin{Rem} For the Muskat-Mixing problem, we shall prove at the end of section \ref{sec:Proof} that \textit{(iii)} holds for all open $\Omega$ instead of only cubes $Q$ with ration extreme points (for $\mapy(t)=\mathrm{id}_{\InterphaseMix(t)}$).
\end{Rem}

\section{Proof of the main results}\label{sec:Proof}

In this section we prove theorems \ref{thmprincipal} and \ref{thmprincipal2} as particular cases of theorem \ref{HP}. For the Muskat-Mixing problem, let us fix $0<c<2$ and $f_0\in\Hil{5}(\R)$. By virtue of \cite[thm. 4.1]{Mixing}, there is $f\in\Cont{0}([0,T];\Hil{4}(\R))$  solving a suitable Cauchy problem for $f_0$ (see \cite[(1.11)]{Mixing}). We consider the map $\map$ associated to $f$ and $c$ and we define the coarse-grained density $\breve{\rho}$ adapted to it \eqref{densidadprincipal} and $\breve{\velocity}=\BSO(-\partial_{1}\breve{\rho})$ (see \cite[(4.13)]{Mixing}). Thus, the domain of perturbation is
$$\DomainP(t)=\InterphaseMix(t),\quad t\in(0,T].$$
The divergence-free expression \eqref{LCproblem1} of the relaxation of the Muskat-Mixing problem is
\begin{equation}\label{CanonicalMuskat}
\Div(\TT z)=\Div\left(\begin{array}{ccc}
\Velocity_1 & \Velocity_2-\rho & 0 \\
\Velocity_2+\rho & -\Velocity_1 & 0 \\
\Aux_1 & \Aux_2 & \rho
\end{array}\right)=0
\end{equation}
in $\R^2\times(0,T)$. From this, it is clear that the associated wave cone is
\begin{equation}\label{LambdaIPM}
\Lambda_{\TT}=\{(\bar{\rho},\bar{\Velocity},\bar{\Aux})\in\R^5\,:\,
|\bar{\rho}|=|\bar{\Velocity}|\}.
\end{equation}
This was already observed in \cite{LIPM,ActiveScalar,RIPM}. For \textrm{(H1)}, we take $\Lambda=\Lambda_{\TT}$. 
For convenience, we briefly recall how it is proved in \cite{LIPM} to be sure that the directions are not of the form $(0,\xi_0)$. 
From \eqref{CanonicalMuskat} it is natural to consider the potential
$$\mathbf{P}(\phi,\varphi)=\left(\begin{array}{ccc}
2\partial_{12}\phi & (\partial_{22}-\partial_{11})\phi-\Delta\phi & 0 \\
(\partial_{22}-\partial_{11})\phi+\Delta\phi & -2\partial_{12}\phi & 0 \\
-\partial_{t1}\phi-\partial_2\varphi & -\partial_{t2}\phi+\partial_1\varphi & \Delta\phi
\end{array}\right),$$
for $\phi,\varphi\in\Cont{3}(\R^{3})$ (notice
$\Div\mathbf{P}(\phi,\varphi)=0$).
Let $\bar{z}=(\bar{\rho},\bar{\Velocity},\bar{\Aux})\in\Lambda$ and $0\neq h\in\Cont{1}(\T;[-1,1])$ with $\int h=0$. Take $H\in\Cont{3}(\T;[-1,1])$ such that $H''=h$. Let $\xi=(\zeta,\xi_0)\in\Sp^1\times\R$ and $a,b\in\R$ to be determined. Consider
$$\phi_k(y)=\tfrac{a}{k^2}H(k\xi\cdot y)
,\quad\varphi_k(y)=\tfrac{b}{k}H'(k\xi\cdot y).$$
Then,
$$\mathbf{P}(\phi_k,\varphi_k)=\left(\begin{array}{ccc}
2a\zeta_1\zeta_2 & -2a\zeta_1^2 & 0 \\
2a\zeta_2^2 & -2a\zeta_1\zeta_2 & 0 \\
-a\zeta_1\xi_0-b\zeta_2 & -a\zeta_2\xi_0+b\zeta_1 & a
\end{array}\right)h(k\xi\cdot y).$$
For our purpose, take $a=\bar{\rho}$. Since $\bar{z}\in\Lambda$, if $\bar{\rho}=0$, then $\bar{z}=(0,0,\bar{\Aux})$. Hence, for this $\Lambda$-direction we take $b=|\bar{\Aux}|$ and $\zeta\in\Sp^1$ such that $b\zeta^\perp=\bar{\Aux}$. Suppose now that $\bar{\rho}\neq 0$. Then, there is $\eta\in\Sp^1$ such that $\bar{\Velocity}=\bar{\rho}\eta$. This induces to take
$\zeta_1=\sqrt{\frac{1-\eta_2}{2}}$ and $
\zeta_2=\sqrt{\frac{1+\eta_2}{2}}$ 
(notice 
$2\zeta_1^2=1-\eta_2$, $2\zeta_2^2=1+\eta_2$ and $2\zeta_1\zeta_2=\eta_1$). Finally, since
$a\neq 0$, we can take $(\xi_0,b)\in\R^2$ solving
$$\left(\begin{array}{cc}
-a\zeta_1 & -\zeta_2 \\
-a\zeta_2 & \zeta_1
\end{array}\right)
\left(\begin{array}{c}
\xi_0 \\ b
\end{array}\right)=
\left(\begin{array}{c}
\bar{\aux}_1 \\ \bar{\aux}_2
\end{array}\right).$$
Therefore, for every $\psi\in\Cont{\infty}_c(\R^{3})$, we obtain a localized plane-wave solution $\tilde{z}_k$ 
$$\mathbf{P}(\psi\phi_k,\psi\varphi_k)
=\TT(\bar{z}h(k\xi\cdot y)\psi+\mathcal{O}(k^{-1}))=\TT \tilde{z}_k$$
where $\mathcal{O}$ only depends on $|\bar{z}|$, $|\xi|$ and $\{|D^\alpha\psi(y)|\,:\,1\leq|\alpha|\leq 2\}$.\\ 

\indent For \textrm{(H2)}, following \cite{RIPM}, we set
$$\tilde{\Kconstrain}=\Kconstrain^\Lambda,\quad
U_M=(\Kconstrain_M^\Lambda)^\circ.$$
By virtue of \cite[prop. 2.4]{RIPM}, $\Leb{\infty}_{\BSO}(\R^2;\Kconstrain^\Lambda)$ is (weak*) closed.
For convenience, we prove a more precise version of \cite[prop. 3.3]{RIPM}.

\begin{lema}\label{lemamejorado} Let $M>2$. Consider the semistrongly concave function $\Dfunc(z)=1-\rho^2$ (with $\Gfunc(z)=-2(\rho,0,0)$ and $\Hfunc(z)=\rho^2$) for $z=(\rho,\Velocity,\Aux)$. Then, there is an increasing function $\Phi_M\in\Cont{}((0,1];(0,\infty))$ so that, for all $z\in U_M$ there is a sizeable $\Lambda$-direction $\bar{z}=(\bar{\rho},\bar{\velocity},\bar{\aux})\in\Lambda$ $$\bar{\rho}^2\geq\Phi_M(1-\rho^2)$$
while the associated $\Lambda$-segment stays in $U_M$
$$z+[-\bar{z},\bar{z}]\Subset U_M.$$
\end{lema} 
\begin{proof} Let $z=(\rho,\Velocity,\Aux)\in U_M$. We want to find a suitable $\bar{z}=(1,e,\bar{\Aux})\in\Lambda$ such that 
	$$z_\lambda=z+\lambda\bar{z}\in U_M$$
	for all $|\lambda|^2\leq \Phi_M(1-\rho^2)$. For \eqref{Khull:1}, it is easy to show that there is an universal constant $0<c_0<\frac{1}{2}$ such that
	$$\tfrac{1}{2}\leq\frac{1\pm\rho_\lambda}{1\pm\rho}\leq 2$$
	for all $|\lambda|\leq c_0(1-\rho^2)$. Thus, we just need to control the other conditions. First assume that $|\Velocity|\geq|\rho|$. Denote 
	$$\Aux_e=\tfrac{1}{2}\rho\Velocity+\tfrac{1}{2}(1-\rho^2)e,
	\quad e\in\Sp^1.$$
	Consider the compact set
	$$\Theta(\rho,\Velocity)=\left\{e\in\Sp^1\,:\,\left|\Aux_e\pm\tfrac{1}{2}\Velocity\right|\leq\tfrac{M}{2}(1\pm\rho)\right\}.$$
	Since
	\begin{align*}
		\left|\Aux_e\pm\tfrac{1}{2}\Velocity\right|^2&=
		\tfrac{1}{4}(1\pm\rho)^2|\Velocity|^2+\tfrac{1}{4}(1-\rho^2)^2\pm
		\tfrac{1}{2}(1\pm\rho)(1-\rho^2)\Velocity\cdot e,
	\end{align*}
	the direction $e\in\Sp^1$ belongs to $\Theta$
	if and only if
	$$\mp2(1\mp\rho)(\rho
	-\Velocity\cdot e)\leq	M^2-|\Velocity|^2-(1-\rho^2).$$
	Since $|\Velocity|\geq|\rho|$, there is $e\in\Sp^1$ such that $\Velocity\cdot e=\rho$, so $\Theta\neq\emptyset$.
	In particular, for every $e\in\Theta$, since
	$(1-\rho^2)|\rho-\Velocity\cdot e|\leq M^2-|\Velocity|^2-(1-\rho^2),$
	for \eqref{Khull:3} we have
	\begin{align*}	M^2-|\Velocity_\lambda|^2-(1-\rho_\lambda^2)
		&=M^2-|\Velocity|^2-(1-\rho^2)+2\lambda(\rho-\Velocity\cdot e)\\
		&\geq(1-2c_0)(M^2-|\Velocity|^2-(1-\rho^2)).
	\end{align*}
	Consider the relative-error points
	$$\omega=\tfrac{2}{1-\rho^2}\left(\Aux-\tfrac{1}{2}\rho\Velocity\right),\quad\omega_{\pm}=\tfrac{2}{M(1\pm\rho)}\left(\Aux\pm\tfrac{1}{2}\Velocity\right),$$
	which satisfy $|\omega|,|\omega_{\pm}|<1$.
	Since $\Theta$ is non-empty, we can select the direction $e\in\Theta$ minimizing $|\Aux-\Aux_e|=\frac{1}{2}(1-\rho^2)|e-\omega|$.
	Now, straightforward computations yield
\begin{align}
\label{H2error1}\Aux_\lambda-\tfrac{1}{2}\rho_\lambda\Velocity_\lambda&=
\tfrac{1}{2}(1-\rho_\lambda^2)\omega+\lambda\left(\bar{\Aux}-\bar{\Aux}_0\right)+\tfrac{1}{2}\lambda^2(\omega-e),\\
\label{H2error2}\Aux_\lambda\pm\tfrac{1}{2}\Velocity_\lambda&=
\tfrac{M}{2}(1\pm\rho_\lambda)\omega_{\pm}+\lambda\left(\bar{\Aux}-\bar{\Aux}_{\pm}\right),
\end{align}
where
$$\bar{\Aux}_0=\tfrac{1}{2}(\Velocity+\rho e)-\rho\omega,
\quad
\bar{\Aux}_{\pm}=\mp\tfrac{1}{2}e\pm\tfrac{M}{2}\omega_{\pm}.$$
First assume $|\omega|>|\omega_+|\vee|\omega_-|$. Since $M>2$, it is not difficult to show that $e=\frac{\omega}{|\omega|}$.
Take
$$\bar{\Aux}=\bar{\Aux}_0.$$
On the one hand, by \eqref{H2error1}, \eqref{Khull:2} holds if and only if
\begin{align*}\frac{|e-\omega|^2}{1-|\omega|^2}\left(\frac{\lambda^2}{1-\rho^2}\right)^2
-2\frac{1-\rho_\lambda^2}{1-\rho^2}\frac{(e-\omega)\cdot \omega}{1-|\omega|^2}\frac{\lambda^2}{1-\rho^2}<\left(\frac{1-\rho_\lambda^2}{1-\rho^2}\right)^2.
\end{align*}
Hence, there is an universal constant $0<c_1\leq c_0$ such that the above inequality holds for all $|\lambda|\leq c_1(1-\rho^2)$.
On the other hand, since
\begin{equation}\label{lemamejorado1}
\bar{\Aux}_0-\bar{\Aux}_{\pm}		
=\tfrac{1}{2}(\rho\pm 1)(e-\omega),
\end{equation}
by \eqref{H2error2}, \eqref{Khull:4} and \eqref{Khull:5} hold if and only if
\begin{align*}
\frac{|e-\omega|^2}{1-|\omega_{\pm}|^2}\lambda^2\pm 2M\frac{1\pm\rho_\lambda}{1\pm\rho}\frac{(e-\omega)\cdot\omega_{\pm}}{1-|\omega_{\pm}|^2}\lambda
<M^2\left(\frac{1\pm\rho_\lambda}{1\pm\rho}\right)^2.
\end{align*}
Since $|e-\omega|=1-|\omega|\leq 1-|\omega_{\pm}|$, there is an universal constant $c>0$ such that the above inequality holds uniformly for $|\lambda|\leq c$.
Otherwise, for some $\pm$ we have $|\omega_{\pm}|\geq|\omega|\vee|\omega_{\mp}|$. Take
$$\bar{\Aux}=\bar{\Aux}_{\pm}.$$
On the one hand, the $\pm$ condition of \eqref{Khull:4} and \eqref{Khull:5} is trivially checked because \eqref{H2error2} implies
$$\left|\Aux_\lambda\pm\tfrac{1}{2}\Velocity_\lambda\right|
=\tfrac{M}{2}(1\pm\rho_\lambda)|\omega_{\pm}|.$$
On the other hand, since
\begin{align*}&\bar{\Aux}_{\pm}-\bar{\Aux}_{\mp}
=\mp(e-\omega),
\end{align*}
by \eqref{H2error2}, the $\mp$ condition of \eqref{Khull:4} and \eqref{Khull:5} holds if and only if
\begin{equation}\label{H2error3}
\frac{|e-\omega|^2}{1-|\omega_{\mp}|^2}\left(\frac{\lambda}{1\mp\rho}\right)^2\mp M\frac{1\mp\rho_\lambda}{1\mp\rho}\frac{(e-\omega)\cdot\omega_{\mp}}{1-|\omega_{\mp}|^2}\frac{\lambda}{1\mp\rho}<\frac{M^2}{4}\left(\frac{1\mp\rho_\lambda}{1\mp\rho}\right)^2.
\end{equation}	
Finally, since \eqref{lemamejorado1}, by \eqref{H2error1} the condition \eqref{Khull:2} is equivalent with
\begin{equation}\label{H2error4}
\frac{|e-\omega|^2}{1-|\omega|^2}\left(\frac{\lambda^2+(\rho\pm1)\lambda}{1-\rho^2}\right)^2-2\frac{1-\rho_\lambda^2}{1-\rho^2}\frac{(e-\omega)\cdot \omega}{1-|\omega|^2}\frac{\lambda^2+(\rho\pm 1)\lambda}{1-\rho^2}<\left(\frac{1-\rho_\lambda^2}{1-\rho^2}\right)^2.
\end{equation}
Both \eqref{H2error3} and \eqref{H2error4} can be verified for $|\lambda|$ depending on $M$ and $1-\rho^2$. More precisely,
for \eqref{H2error4} consider $N=\{(\rho,\vect,\Aux)\in U_M\,:\, e\cdot\omega\geq|\omega|^2\textrm{ for some }e\in\Theta(\rho,\vect)\}$. Notice $e\cdot\omega\geq|\omega|^2$ defines the circumference of radius $\tfrac{1}{2}$ inside the unit-circumference which is tangential to $e$. If $z\in N$ we have $|e-\omega|^2\leq 1-|\omega|^2$ and also $|(e-\omega)\cdot\omega|\leq|e-\omega|^2+(1-e\cdot\omega)\leq 2(1-|\omega|^2)$, while outside it is not difficult to prove that $|\omega|\ll 1$ due to $M>2$. For \eqref{H2error3}, it is not difficult to show that there is $C$ depending continuously on $M$ and $1-\rho^2$ such that $C|e-\omega|\leq(1-|\omega_{\mp}|)$.\\
Finally, if $|\Velocity|\leq|\rho|$, notice that the condition \eqref{Khull:3} is easily checked
	\begin{align*}M^2-|\Velocity_\lambda|^2-(1-\rho_\lambda^2)
	\geq
		M^2-1-4|\lambda|\geq\tfrac{1}{2}(M^2-1)
	\end{align*}
	for $|\lambda|\leq\frac{1}{8}(M^2-1)$. Then, by taking $e\in\Sp^1$ minimizing $|e-\omega|$ and $\bar{\Aux}$ as before, we check the other conditions analogously.
\end{proof}

\indent For \textrm{(H3)}, it is shown in \cite{Mixing} that there exists $\breve{\Aux}$ such that $\breve{z}=(\breve{\rho},\breve{\Velocity},\breve{\Aux})$ becomes into a $\tilde{\Kconstrain}$-subsolution. We surround $\breve{z}$ in the topological space $\Sub$ of admissible perturbations (def. \ref{spacesubsolutions}). For this, we fix a finite family $\Fspace_0\subset\Fspace$ given by the triples $(\FF,1,\mapy)$:
\begin{enumerate}[1)]
\item\label{F01} $\FF(z)=\rho$ and $\mapy=\map$ with $0<\alpha<\tfrac{1}{r}=1$,
\item\label{F03} $\FF(z)=\velocity$ and $\mapy=\map$ with $0<\alpha<\tfrac{1}{r}=1$,
\item\label{F04} $\FF=\PP$ \eqref{Power} and $\mapy=\map$ with $0<\alpha<\tfrac{1}{r}=1$,
\item\label{F02} $\FF(z)=1\mp\rho$ and $\mapy(t)=\mathrm{id}_{\InterphaseMix(t)}$.
\end{enumerate}

Therefore, we can apply theorem \ref{HP} to the Muskat-Mixing problem. With $\ref{F01}$, we deduce the ``linearly degraded macroscopic behaviour'' property (def. \ref{degradada}\ref{DMS3}). With $\ref{F02}$, since for every $z\in\Sub$ 
$$\int_Q(1\mp\rho(x,t))\dif x\notin\{0\}$$
for every cube $Q\subset\InterphaseMix(t)$ with rational extreme points at each $t\in(0,T]$, necessarily, for every $z\in\SSub$, we have
$$\int_{\Omega}(1\mp\rho(x,t))\dif x\neq 0$$
for every non-empty bounded open  $\Omega\subset\InterphaseMix(t)$ at each $t\in(0,T]$.
Otherwise, it would be $\rho|_{\Omega}=\pm 1$ ($|\rho|\leq 1$ for states in $\SSub$). But then, since $\Omega$ is open, we would find an open cube $Q\subset\Omega$ with rational extreme points, which would contradict corollary \ref{corHP}. This is exactly the ``mix in space at each time slice'' property (def. \ref{degradada}\ref{DMS2}). Finally, we prove theorem \ref{thmprincipal2} with $\ref{F03}$ and $\ref{F04}$. For that we recall that $\PP$ is continuous from $\Leb{\infty}_{\BSO}$ to $\Leb{\infty}_{w^*}$ on bounded subsets of $\R^2$ as a consequence of the div-curl lemma (\cite{Tartar}).


\section{Application to the vortex sheet problem}\label{sec:Other}

The motion of an ideal incompressible fluid is modelled by the \textbf{incompressible Euler equations}
\begin{align}
\partial_t\velocity+\Div(\velocity\otimes\velocity)+\nabla p & = 0, \label{Euler:1}\\
\Div\velocity & = 0,\label{Euler:2}
\end{align}
in $\R^2\times(0,T)$, where $\velocity$ is the incompressible velocity field and $p$ is the pressure. As for the Muskat-Mixing problem, $p$ may be ignored. As pointed in the ground-breaking work of De Lellis and Sz\'ekelyhidi Jr. (\cite{Eulerinclusion}), the incompressible Euler equations can be shown as a differential inclusion. In addition, as we have commented, they introduced in \cite{Onadmissibility} the ideas for the convex integration scheme that we have adapted for other evolution equations in the section \ref{sec:CI}. On the one hand, the space associated to the system of the stationary equation \eqref{Euler:2} is $\Leb{\infty}_{\mathrm{div}}(\R^2)$. On the other hand, the Cauchy problem is given by \eqref{Euler:1} for some initial data $\velocity_0\in\Leb{\infty}_{\mathrm{div}}(\R^2)$. A similar situation to the Muskat-Mixing problem with $f_0=0$ for these equations is given by the vortex sheet initial data 
\begin{equation}\label{VSinitialdata}
\velocity_0(x)=(\car{\Omega_{+}(0)}-\car{\Omega_{-}(0)})\mathbf{e}_1
\end{equation}
where $\mathbf{e}_1=(1,0)$ and $\Omega_{\pm}(0)=\{x\in\R^2\,:\, \pm x_2>0\}$.
In \cite{vortexsheetIEE}, Sz\'ekelyhidi Jr. applied the convex integration method to prove the existence of infinitely many weak solutions in $\Cont{0}([0,T];\Leb{\infty}_{\mathrm{div}}(\R^2))$ for the vortex sheet initial data \eqref{VSinitialdata} for which a
``\textbf{turbulence zone}'' defined by
$$\Omega_{\textrm{tur}}(t)=\{x\in\R^2\,:\, |x_2|<ct\}$$
appears, where $0<c<1$ represents its speed of growth. Similarly, the distinguished regions are  $\Omega_{\pm}(t)=\{x\in\R^2\,:\, \pm x_2>ct\}$.
For a suitable relaxation (\cite{Onadmissibility}), the subsolution is 
$\breve{\velocity}(x,t)=u(x,t)\mathbf{e}_1$ with
$$u(x,t)=\left\lbrace
\begin{array}{rl}
\pm 1, & \pm x_2>ct,\\[0.1cm]
\frac{x_2}{ct}, & |x_2|<ct.
\end{array}\right.$$
Now, the quantitative h-principle allows to select those solutions which best inherit the properties of the subsolution. More precisely, the following theorem holds.

\begin{thm} Let $\Efunc$ from \eqref{Efuncdef} and $\alpha\in[0,1)$. There exists infinitely many weak solutions $\velocity\in\Cont{}([0,T];\Leb{\infty}_{\mathrm{div}}(\R^2))$ to the incompressible Euler equations for the vortex sheet initial data \eqref{VSinitialdata} satisfying: the modulus of the velocity is constant $|\velocity|=1$, they are not affected outside the turbulence zone
$$\velocity(\cdot,t)|_{\Omega_{\pm}(t)}=\pm \mathbf{e}_1,$$
while the behaviour inside the turbulence zone obeys
$$\int_{\Omega}(1-\velocity_2(x,t))\dif x\int_{\Omega}(1+\velocity_2(x,t))\dif x\neq0$$
for every non-empty bounded open  $\Omega\subset\Omega_{\mathrm{tur}}(t)$, but displaying a linearly degraded macroscopic behaviour
$$\left|\dashint_Q\velocity(x,t)\dif x-\expected{L}\mathbf{e}_1\right|
\leq\Efunc(\expected{L},t)\Qfunc{Q}$$
for every non-empty bounded rectangle $Q=S\times ctL\subset\R\times(-ct,ct)$ at each $t\in(0,T]$.
\end{thm}

\section*{Acknowledgements}

AC and DF were partially supported by ICMAT Severo Ochoa projects SEV-2011-0087 and
SEV-2015-556, and by the ERC grant 307179-GFTIPFD. 
DF and FM were partially supported by the grant MTM2017-85934-C3-2-P (Spain).
AC were partially supported by the grant MTM2014-59488-P (Spain) and
DF by the grant MTM2014-57769-P-1 and MTM2014-57769-P-3 (Spain). FM were partially supported by ICMAT Severo Ochoa project SEV-2015-0554 with a FPI predoctoral research grant of MINECO (Spain).

\centering

\begin{flushleft}
\quad\\
\textbf{\'Angel Castro}\\
Instituto de Ciencias Matem\'aticas-CSIC-UAM-UC3M-UCM\\
Email: angel\_castro@icmat.es\\[0.3cm]
\textbf{Daniel Faraco}\\
Departamento de Matem\'aticas\\
Universidad Aut\'onoma de Madrid\\
Instituto de Ciencias Matem\'aticas-CSIC-UAM-UC3M-UCM\\
Email: daniel.faraco@uam.es\\[0.3cm]
\textbf{Francisco Mengual}\\
Departamento de Matem\'aticas\\
Universidad Aut\'onoma de Madrid\\
Instituto de Ciencias Matem\'aticas-CSIC-UAM-UC3M-UCM\\
Email: francisco.mengual@estudiante.uam.es
\end{flushleft}

\end{document}